\documentclass[12 pt,english,oneside]{amsart}
\usepackage[top=1.26in, bottom=1.26in, left=1.33in, right=1.33in]{geometry}
\usepackage[latin1]{inputenc}
\usepackage[T1]{fontenc}
\usepackage[normalem]{ulem}
\usepackage{verbatim} 
\usepackage{graphicx}
\usepackage{amsfonts}
\usepackage{amsmath}
\usepackage{amssymb}
\usepackage{graphicx,epsfig}
\usepackage{amsthm}
\usepackage{amscd}
\usepackage{dsfont}
\usepackage{fancyhdr}

\newcommand{\g}{\mathfrak{g}}

\newcommand{\uqg}{\mathcal{U}_q(\mathfrak{g})}

\newcommand{\ra}{\rightarrow}

\newcommand{\ve}{\varepsilon}
\newcommand{\vp}{\varphi}

\newcommand{\ts}{\otimes}
\newcommand{\s}{\sigma}

\newcommand{\pp}{\mathcal{P}}
\newcommand{\mz}{\mathbb{Z}}
\newcommand{\mc}{\mathbb{C}}
\newcommand{\fb}{\mathfrak{B}}
\newcommand{\fn}{\mathfrak{N}}
\newcommand{\p}{\partial}

\newtheorem{definition}{Definition}
\newtheorem{proposition}{Proposition}
\newtheorem{theorem}{Theorem}
\newtheorem{lemma}{Lemma}
\newtheorem{corollary}{Corollary}
\newtheorem{remark}{Remark}
\newtheorem{example}{Example}

\title{On defining ideals and differential algebras of Nichols algebras}
\author{Xin Fang}
\address{Universit\'{e} Paris 7, Institut de Math\'{e}matiques de Jussieu, Th\'{e}orie des groupes et des repr\'{e}sentations, Case 7012, 2 place Jussieu, 75251 Paris Cedex 05, France.}
\email{fang@math.jussieu.fr}

\begin{document}

\maketitle

\begin{abstract}
This paper is devoted to understanding the defining ideal of a Nichols algebra from the decomposition of specific elements in the group algebra of braid groups. A family of primitive elements are found and algorithms are proposed. To prove the main result, the differential algebra of a Nichols algebra is constructed. Moreover, another point of view on Serre relations is provided.
\end{abstract}

\section{Introduction}
Nichols algebras, as their name says, are constructed by W.D. Nichols in \cite{Nic78} with the name "bialgebra of type one" for classifying finite dimensional graded Hopf algebras generated by elements in degree $0$ and $1$. No much attention was paid to his work at that time until quantized enveloping algebras are constructed by Drinfel'd and Jimbo in the middle of eighties.
\par
The construction of Nichols, after having slept for about 15 years, is highlighted by M. Rosso in his article \cite{Ros98} to give a functorial and coordinate free construction of quantized enveloping algebras, which meanwhile gives the motivation and another point of view to researches on pointed Hopf algebras.
\par
The construction in \cite{Ros98} starts with a Hopf algebra $H$ and an $H$-Hopf bimodule $M$. Then the set of right coinvariants $M^R$ admits a braiding $\sigma_M:M^R\ts M^R\ra M^R\ts M^R$. Once the usual flip is replaced by this braiding, the classical construction of shuffle algebra will give a new Hopf algebra whose structure is controlled by this braiding, and is called quantum shuffle algebra. This construction gives many interesting examples: the positive part of a quantized enveloping algebra associated to a symmetrizable Cartan matrix can be found as a Hopf sub-algebra generated by $H$ and $M^R$ in the quantum shuffle algebra with some particular Hopf algebra $H$ and $H$-Hopf bimodule $M$.
\par
A well-known result affirms that there is an equivalence of braided category between the category of Hopf bimodules over $H$ and that of left $H$-Yetter-Drinfel'd modules, given by sending a Hopf bimodule $M$ to the set of its right coinvariants $M^R$; so it is possible to work in the context of Yetter-Drinfel'd modules at the very beginning. This gives a translation of language between Nichols-Rosso and Andruskiewitsch-Schneider.
\par
The dual construction of quantum shuffle algebra is easier to understand: it is the braided tensor algebra $T(V)$ for $V=M^R$ and the Nichols algebra $\mathfrak{N}(V)$ is the quotient of $T(V)$ by some Hopf ideal $\mathfrak{I}(V)$. As an example, it gives the strict positive part of quantized enveloping algebras when $H$ and $V$ are properly chosen.
\par
In the latter case, as a quotient of the braided Hopf algebra $T(V)$, $\mathfrak{N}(V)$ can be viewed as imposing some relations in $T(V)$ which have good behavior under the coproduct. It is natural to ask for the structure of such ideals, but unfortunately, this Hopf ideal is defined as the maximal coideal contained in the set of elements in $T(V)$ of degree no less than $2$ and it is very difficult to read out these relations directly from such an abstract definition.
\par
Similar problems arise in some other places in mathematics. For example, the Gabber-Kac theorem in the theory of Kac-Moody Lie algebras is of the same philosophy: starting with a symmetrizable Cartan matrix and some Chevalley generators, it is possible to construct a Lie algebra with Chevalley relations; but to get a Kac-Moody Lie algebra, one is forced to do the quotient by an ideal with some maximal properties and of course with mysteries. It is in Gabber-Kac \cite{GK81} that they proved that this ideal is generated by Serre relations, which completes the whole story. It should be remarked that this result is not simple at all: the proof clarifies some structures and uses several tools in Kac-Moody Lie algebras, such as generalized Casimir elements and Verma modules.
\par
As we know, for Nichols algebras, the problem of deciding generating relations in the ideal $\mathfrak{I}(V)$ is still open and the best general result can be found in the literature is due to M. Rosso \cite{Ros98}, P. Schauenburg \cite{Sbg96} and others, which affirms that elements of degree $n$ in $\mathfrak{I}(V)$ are those annihilated by the symmetrization operator $S_n$. However, it is still difficult to make these relations explicit because of the appearance of large amount of equations when a basis is chosen.
\par
The main objective of this paper is to give some restrictions on elements in $\mathfrak{I}(V)$ to obtain a number of important relations imposed, and we conjecture that these relations generate $\mathfrak{I}(V)$ as an ideal.
\par
To be more precise, the restriction is at first given by some operators $P_n$, called Dynkin operators, on $T(V)$; these operators can be viewed as analogues of the Dynkin projection operators in the characterization of free Lie algebras. Our first restriction is passing from $\ker(S_n)$ to $\ker(S_n)\cap Im(P_n)$: though the latter is somehow subtle at a first glimpse, in some important cases, they will generate the Hopf ideal $\mathfrak{I}(V)$; in general, if a conjecture of Andruskiewitsch-Schneider is true, the statement above holds for any Nichols algebras. This is proved by showing that all primitive elements of degree $n$ are eigenvectors of $P_n$ with eigenvalue $n$.
\par
Another restriction is given by concentrating on some levels in $T(V)$ having their origin in the decomposition of the element $S_n$ in the group algebra $\mc[\mathfrak{B}_n]$, where $\fb_n$ is the braid group on $n$ strands. The main idea here is building a bridge to connect some solutions of equation $S_nx=0$ in $V^{\ts n}$ with the invariant space for the central element $\theta_n$ of $\fb_n$, which are much easier to understand and compute. Moreover, it throws some light on understanding the structure of Nichols algebras from the representation theory of braid groups, though the latter is difficult indeed.
\par
When constructing this bridge, we captured the appearance of the Dynkin operator $P_n$ as an important ingredient. Moreover, it is a central tool when constructing solutions of $S_nx=0$ from solutions of the equation $\theta_nx=x$. As an example, we obtain quantized Serre relations from a pedestrian point of view: that is to say, the machinery will tell us what these relations are by assuming almost no knowledge on Lie theory and quantized enveloping algebra.
\par
A natural question is posed when observing the calculation in the exterior algebra and quantized enveloping algebras: whether elements with levels in $T(V)$ are primitive?
\par
The second part of this paper is devoted to give a positive answer in the general case: that is to say, for any braiding coming from a Yetter-Drinfel'd structure.
\par
The proof is based on the construction of the differential algebra of a Nichols algebra, which can be viewed as a generalization of the construction of the quantized Weyl algebra over a quantized enveloping algebra given in \cite{Fang10}. This construction, when restricted to Nichols algebras of diagonal type, gives a generalization for the construction of skew-derivations dates back to N.D. Nichols \cite{Nic78}, Section 3.3 and explicitly used in M. Gra\~na \cite{Gra00} and I. Heckenberger \cite{Hec06}, \cite{Hec08}. The advantage of our approach is: at the very beginning, we never make hypotheses on the type of the braiding, so this is a general construction shared by all kinds of Nichols algebras coming from a Yetter-Drinfel'd module.
\par
Once restricting ourselves to the diagonal case, with the help of these differential operators, we proved the classical Taylor lemma, which generalizes a result in Heckenberger \cite{Hec08}. Moreover, when the derivation is given by a primitive element of degree $1$, a decomposition theorem of $T(V)$ as a braided algebra is obtained, which can be viewed as a generalization of a result in classical Weyl algebra given by A. Joseph \cite{Jos77} in solving the Gelfan'd-Kirillov conjecture.
\par
At last, we show that elements with levels are all primitive with the help of these differential operators.\\
\par
The organization of this paper is as follows:
\par
In Section 2, some notions in Hopf algebras are recalled and notations are introduced. Section 3 is devoted to defining Dynkin operators and proving the "convolution invariance" of these operators. In Section 4, we are concerned with the decomposition of specific elements in the group algebra of braid groups, which is an algebraic preparation for solving equation $S_nx=0$. The construction of the bridge mentioned above is given in Section 5 and some properties of $\ker(S_n)\cap Im(P_n)$ are obtained. Section 6 contains examples in the diagonal case; a concrete calculation of quantized Serre relations in the case $\mathcal{U}_q(\mathfrak{sl}_3)$ is among examples. In Section 7, the differential algebra of a Nichols algebra is constructed and then the Taylor Lemma and a decomposition theorem are proved as an application in Section 8. Finally, the main theorem on primitive elements is demonstrated with the help of the differential algebra in Section 9.\\
\\
\noindent
\textbf{Acknowledgements.}
I am grateful to my advisor Marc Rosso for the discussion and treasurable remarks on this work. I would like to thank Victoria Lebed for her comments on an early version of this paper. I want to thank the referee for preventing me from mistakes and simplifying the proof of Theorem 1.

\section{Recollections on Hopf algebras}\label{section2}
From now on, suppose that we are working in the complex field $\mathbb{C}$. All algebras and modules concerned are over $\mathbb{C}$. Results in this section will hold for any field with characteristic $0$. All tensor products are over $\mathbb{C}$ if not specified otherwise.
\par
This section is devoted to giving a recollection on some constructions in Hopf algebras and fixing notations.

\subsection{Yetter-Drinfel'd modules}
Let $H$ be a Hopf algebra. A vector space $V$ is called a (left) $H$-Yetter-Drinfel'd module if:
\begin{enumerate}
\item It is simultaneously
an $H$-module and an $H$-comodule;
\item These two structures satisfy the Yetter-Drinfel'd compatibility condition: for any
$h\in H$ and $v\in V$,
$$\sum h_{(1)}v_{(-1)}\ts h_{(2)}.v_{(0)}=\sum (h_{(1)}.v)_{(-1)}h_{(2)}\ts (h_{(1)}.v)_{(0)},$$
where $\Delta(h)=\sum h_{(1)}\ts h_{(2)}$ and $\rho(v)=\sum v_{(-1)}\ts v_{(0)}$ are  Sweedler notations for coproduct and comodule structure map.
\end{enumerate}
\indent
Morphisms between two $H$-Yetter-Drinfel'd modules are linear maps preserving $H$-module
and $H$-comodule structures.
\par
We denote the category of $H$-Yetter-Drinfel'd modules by ${}_H^H\mathcal{YD}$; it is a tensor category.
\par
The advantage of working in the category of Yetter-Drinfel'd module is: for $V,W\in {}_H^H\mathcal{YD}$, there exists a braiding $\sigma_{V,W}:V\ts W\ra W\ts V$, given by $\sigma_{V,W}(v\ts w)=\sum v_{(-1)}.w\ts v_{(0)}$. This braiding gives ${}_H^H\mathcal{YD}$ a braided tensor category structure.
\par
Let $A$ and $B$ be two $H$-module algebras in ${}_H^H\mathcal{YD}$. Then the vector space $A\ts B$ admits an algebra structure with the following multiplication map:
\[
\begin{CD}
A\ts B\ts A\ts B @>{id\ts\sigma_{B,A}\ts id}>> A\ts A\ts B\ts B  @>{m_A\ts m_B}>> A\ts B,
\end{CD}
\]
where $m_A$ and $m_B$ are multiplications in $A$ and $B$, respectively. We let $A\underline{\ts} B$ denote this algebra.

\subsection{Braided Hopf algebras in ${}_H^H\mathcal{YD}$}
\begin{definition}[\cite{AS02}, Section 1.3]
A braided Hopf algebra in the category ${}_H^H\mathcal{YD}$ is a collection $(A,m,\eta,\Delta,\ve,S)$ such that:
\begin{enumerate}
\item $(A,m,\eta)$ is an algebra in ${}_H^H\mathcal{YD}$; $(A,\Delta,\ve)$ is a coalgebra in
${}_H^H\mathcal{YD}$. That is to say, $m,\eta,\Delta,\ve$ are morphisms in ${}_H^H\mathcal{YD}$;
\item $\Delta:A\ra A\underline{\ts} A$ is an algebra morphism;
\item $\ve:A\ra \mathbb{C}$ is an algebra morphism;
\item $S$ is the convolution inverse of $id_A\in End(A)$.
\end{enumerate}
\end{definition}
The most important example of a braided Hopf algebra is the braided tensor Hopf algebra defined as follows.
\begin{example}[\cite{AS02}]\label{example}
Let $V\in{}_H^H\mathcal{YD}$ be an $H$-Yetter-Drinfel'd module. There exists a braided Hopf algebra structure on the tensor algebra 
$$T(V)=\bigoplus_{n=0}^\infty V^{\ts n}.$$
\begin{enumerate}
\item The multiplication is the original one on $T(V)$ given by the concatenation.
\item The coalgebra structure is defined on $V$ by: for any $v\in V$, $\Delta(v)=v\ts 1+1\ts v$, $\ve(v)=0$. Then it can be extended to the whole $T(V)$ by the universal property of $T(V)$ as an algebra.
\end{enumerate}
\end{example}

\subsection{Nichols algebras}
Let $V\in{}_H^H\mathcal{YD}$ be a finite dimensional $\mc$-vector space and $T(V)$ be the braided tensor Hopf algebra over $V$ as defined in Example \ref{example}; it is $\mathbb{N}$-graded.
\par
We will recall briefly the definition and the explicit construction of Nichols algebras, which dates back to \cite{Nic78} and is given in \cite{AS02}. Another definition in a dual point of view is given in \cite{Ros98} under the name quantum shuffle algebra and is denoted by $S_\s(V)$. The difference between them is: the construction in \cite{Ros98} is in the graded dual of $T(V)$, so instead of being a quotient object, it will be a sub-object in the graded dual. But they are isomorphic as braided Hopf algebra up to a symmetrization morphism.
\begin{definition}[\cite{AS02}]
A graded braided Hopf algebra $R=\bigoplus_{n=0}^\infty R(n)$ is called a Nichols algebra of $V$ if 
\begin{enumerate}
\item $R(0)\cong\mc$, $R(1)\cong V\in{}_H^H\mathcal{YD}$;
\item $R$ is generated as an algebra by $R(1)$;
\item $R(1)$ is the set of all primitive elements in $R$.
\end{enumerate}
We let $\mathfrak{N}(V)$ denote this braided Hopf algebra.
\end{definition}
\begin{remark}\label{ASconjecture}
It is conjectured by Andruskiewitsch and Schneider in \cite{AS02} that when $R$ is finite dimensional, (3) implies (2).
\end{remark}
\indent
There is a construction of $\mathfrak{N}(V)$ from $T(V)$ as shown in \cite{AS02}: let 
$$T^{\geq 2}(V)=\bigoplus_{n\geq 2}V^{\ts n}$$ 
and $\mathfrak{I}(V)$ be the maximal coideal of $T(V)$ contained in $T^{\geq 2}(V)$. Then $\mathfrak{I}(V)$ is also a two-sided ideal; the Nichols algebra $\mathfrak{N}(V)$ of $V$ can be constructed as $T(V)/\mathfrak{I}(V)$. We let $S$ denote the convolution inverse of $id:\mathfrak{N}(V)\ra\mathfrak{N}(V)$.
\par
For $k\in\mathbb{N}$, let $\mathfrak{N}(V)_k$ denote the subspace of degree $k$ elements in $\mathfrak{N}(V)$. From definition, $\mathfrak{N}(V)_0=\mathbb{C}$ and $\mathfrak{N}(V)_1=V$ is the set of all primitive elements in $\mathfrak{N}(V)$.

\subsection{Nichols algebras of diagonal type}\label{diagonal}
In this subsection, we recall a particular type of Nichols algebra, which will be a good source of examples in our later discussions. A concrete approach can be found in \cite{AS02}.
\par
Let $G$ be an abelian group and $H=\mc[G]$ be its group algebra: it is a commutative and cocommutative Hopf algebra. We let $\widehat{G}$ denote the character group of $G$. Let $V\in{}_H^H\mathcal{YD}$ be a finite dimensional $H$-Yetter-Drinfel'd module and $dimV=n$. Let $T(V)$ denote the braided tensor Hopf algebra and $\mathfrak{N}(V)$ denote the associated Nichols algebra.
\par
As shown in \cite{Nic78} or Remark 1.5 in \cite{AS02}, the category ${}_H^H\mathcal{YD}$ is made of a $G$-graded vector space $V=\bigoplus_{g\in G}V_g$ such that for any $h\in G$ and $v\in V_g$, $h.v\in V_g$. The comodule structure is given by: for $V$ in ${}_H^H\mathcal{YD}$ and $v\in V_g$ in the decomposition above, the comodule structure map $\delta:V\ra H\ts V$ is  $\delta(v)=g\ts v$.
\begin{definition}\label{diagonal}
Let $V\in{}_H^H\mathcal{YD}$ be a finite dimensional $H$-Yetter-Drinfel'd module. $V$ is called of diagonal type if there exists a basis $\{v_1,\cdots,v_n\}$ of $V$, $g_1,\cdots,g_n\in G$ and $\chi_1,\cdots,\chi_n\in\widehat{G}$ such that for any $g\in G$ and $v_i\in V_{g_i}$,
$$g.v_i=\chi_i(g)v_i.$$
\end{definition}
Sometimes, we call $T(V)$ of diagonal type if $V$ is of diagonal type.
\par
In this case, the braiding in ${}_H^H\mathcal{YD}$ is given by: for $V,W\in {}_H^H\mathcal{YD}$,
$$\sigma_{V,W}:V\ts W\ra W\ts V,\ \ \sigma_{V,W}(v\ts w)=(g.w)\ts v$$
for any $g\in G$, $v\in V_g$ and $w\in W$.
\par
In particular, if we choose $V=W$ and $v_1,\cdots,v_n$ be a basis of $V$ as in the definition above, the braiding, when acting on basis elements, is given by: for $1\leq i,j\leq n$,
$$\sigma_{V,V}(v_i\ts v_j)=\chi_j(g_i)v_j\ts v_i.$$
So $\sigma_{V,V}$ is completely determined by the matrix $(\chi_j(g_i))_{1\leq i,j\leq n}$. We denote $q_{ij}=\chi_j(g_i)$ and call $(q_{ij})_{1\leq i,j\leq n}$ the braiding matrix associated to $\s_{V,V}$.
\par
It is convenient to define a bicharacter $\chi$ over $G$ when $G=\mathbb{Z}^n$ to rewrite the braiding above.
\begin{definition}
A bicharacter on an abelian group $A$ is a map $\chi:A\times A\ra\mc^*$ such that:
$$\chi(a+b,c)=\chi(a,c)\chi(b,c),\ \ \chi(a,b+c)=\chi(a,b)\chi(a,c),$$
for any $a,b,c\in A$.
\end{definition}
Suppose that $G=\mathbb{Z}^n$ and $V\in{}_H^H\mathcal{YD}$. Then $V$, $T(V)$ and $\mathfrak{N}(V)$ are all $\mathbb{Z}^n$-graded. Let $v_1,\cdots,v_n$ be a basis of $V$ as in Definition \ref{diagonal}, $\alpha_1,\cdots,\alpha_n$ be a free basis of $\mathbb{Z}^n$ and $deg(v_i)=\alpha_i$ be their grade degrees in $\mathbb{Z}^n$.
\par
If this is the case, a bicharacter over $\mathbb{Z}^n$ can be defined using the braiding matrix: for any $1\leq i,j\leq n$, $\chi:\mathbb{Z}^n\times\mathbb{Z}^n\ra \mathbb{C}^*$ is determined by $\chi(\alpha_i,\alpha_j)=q_{ij}$.

\subsection{Radford's biproduct}\label{biproduct}
Let $A\in{}_H^H\mathcal{YD}$ be a braided Hopf algebra. Then $A\ts H$ admits a Hopf algebra structure from a construction in Radford \cite{Rad85}, which is called the biproduct of $A$ and $H$.
\par
These structures are defined by:

\begin{enumerate}
\item The multiplication is given by the crossed product: for $a,a'\in A$, $h,h'\in H$, 
$$(a\ts h)(a'\ts h')=\sum a(h_{(1)}.a')\ts h_{(2)}h';$$
\item The comultiplication is given by: for $a\in A$ and $h\in H$,
$$\Delta(a\ts h)=\sum (a_{(1)}\ts (a_{(2)})_{(-1)}h_{(1)})\ts ((a_{(2)})_{(0)}\ts h_{(2)});$$
\item The antipode is completely determined by: for $a\in A$ and $h\in H$,
$$S(a\ts h)=\sum (1\ts S_H(h)S_H(a_{(-1)}))(S_A(a_{(0)})\ts 1).$$
\end{enumerate}

\begin{proposition}[Radford]
With structures defined above, $A\ts H$ is a Hopf algebra. We let $A\sharp H$ denote it.
\end{proposition}

\subsection{Quantum doubles and Heisenberg doubles}
We recall first the definition of a generalized Hopf pairing, which gives duality between Hopf algebras.
\begin{definition}[\cite{KRT94}]
Let $A$ and $B$ be two Hopf algebras with invertible antipodes. A
generalized Hopf pairing between $A$ and $B$ is a bilinear form
$\vp:A\times B\ra\mathbb{C}$ such that:
\begin{enumerate}
\item For any $a\in A$, $b,b'\in B$,
$\vp(a,bb')=\sum\vp(a_{(1)},b)\vp(a_{(2)},b')$;
\item For any $a,a'\in A$, $b\in B$,
$\vp(aa',b)=\sum\vp(a,b_{(2)})\vp(a',b_{(1)})$;
\item For any $a\in A$, $b\in B$,
$\vp(a,1)=\ve(a),\ \ \vp(1,b)=\ve(b)$.
\end{enumerate}
\end{definition}
\begin{remark}
From the uniqueness of the antipode and conditions (1)-(3) above, we have: for any $a\in A$, $b\in B$,
$\vp(S(a),b)=\vp(a,S^{-1}(b))$.
\end{remark}
\indent
Starting with a generalized Hopf pairing between two Hopf algebras, we can define the quantum double and the Heisenberg double of them, which will be essential in our later construction of differential algebras of Nichols algebras.

\begin{definition}[\cite{KRT94}]
Let $A$, $B$ be two Hopf algebras with invertible antipodes and $\vp$ be
a generalized Hopf pairing between them. The quantum double $D_\vp(A,B)$ is defined by:
\begin{enumerate}
\item As a vector space, it is $A\ts B$;
\item As a coalgebra, it is the tensor product of coalgebras $A$ and $B$;
\item As an algebra, the multiplication is given by:
$$(a\ts b)(a'\ts b')=\sum \vp(S^{-1}(a_{(1)}'),b_{(1)})\vp(a_{(3)}',b_{(3)})aa_{(2)}'\ts b_{(2)}b'.$$
\end{enumerate}
\end{definition}

Then we construct the Heisenberg double
of $A$ and $B$: it is a crossed product of them where the module algebra type action of $A$ on $B$ is given by the Hopf pairing.
\begin{definition}[\cite{Lu94}]
The Heisenberg double $H_\vp(A,B)$ of $A$ and $B$ is an algebra defined as follows:
\begin{enumerate}
\item As a vector space, it is $B\ts A$ and we let $b\sharp a$ denote a pure tensor;
\item The product is given by: for $a,a'\in A$, $b,b'\in B$,
$$(b\sharp a)(b'\sharp a')=\sum \vp(a_{(1)},b_{(1)}')bb_{(2)}'\sharp a_{(2)}a'.$$
\end{enumerate}
\end{definition}

\section{Dynkin operators and their properties}
In this section, we will define Dynkin operators in the group algebras of braid groups. The definition of these operators is motivated by the iterated brackets in Lie algebras which are used by Dynkin in the proof of the Dynkin-Wever-Spechet theorem for characterizing elements in free Lie algebras (for example, see \cite{Reu93}).
\par
As will be shown in this section, Dynkin operators have good properties under the convolution product, which generalizes the corresponding result in free Lie algebras. This will be used to detect primitive elements later.

\subsection{Definition of Dynkin operators}\label{Dynkin}
We suppose that $n\geq 2$ is an integer.
\par
Let $\mathfrak{S}_n$ denote the symmetric group: it acts on an alphabet with $n$ letters by permuting their positions. It can be generated by the set of transpositions $\{s_i=(i,i+1)|\ 1\leq i\leq n-1\}$. 
\par
Let $\mathfrak{B}_n$ be the braid group of $n$ strands, it is defined by generators $\s_i$ for $1\leq i\leq n-1$ and relations:
$$\s_i\s_j=\s_j\s_i,\ \ \text{for}\ |i-j|\geq 2;\ \ \s_i\s_{i+1}\s_i=\s_{i+1}\s_i\s_{i+1},\ \ \text{for}\ 1\leq i\leq n-2.$$ 
Let $\pi:\mathfrak{B}_n\ra\mathfrak{S}_n$ be the canonical surjection which maps $\s_i\in\mathfrak{B}_n$ to $s_i=(i,i+1)\in\mathfrak{S}_n$.
\par
We consider the group $\mathfrak{S}_n^\pm=\mz/2\mz\times\mathfrak{S}_n$, where $\mz/2\mz=\{\pm 1\}$ is the signature.
We are going to define a subset $\pp_{i,j}\subset\mathfrak{S}_n^\pm$ by induction on $|i-j|$ for $1\leq i\leq j\leq n$.
\par
We omit the signature $1$. Define $\pp_{i,i}=\{(1)\}$, $\pp_{i,i+1}=\{(1), -(i,i+1)\}$, and 
$$\pp_{i,j}=\pp_{i+1,j}\cup (\pp_{i,j-1}\circ -(i,j,j-1,\cdots,i+1)),$$
where $\circ$ is the product in $\mc[\mathfrak{S}_n]$.
\par 
Moreover, we define 
$$P_{i,j}=\sum_{(\ve,\omega)\in\mathcal{P}_{i,j}}\ve\omega\in\mathbb{C}[\mathfrak{S}_n].$$
\indent
Let $\sigma\in\mathfrak{S}_n$ and $\sigma=s_{i_1}\cdots s_{i_r}$ be a reduced expression of $\sigma$. It is possible to define a corresponding lifted element $T_\sigma=\s_{i_1}\cdots\s_{i_r}\in\mathfrak{B}_n$. This gives a linear map $T:\mathbb{C}[\mathfrak{S}_n]\ra\mathbb{C}[\mathfrak{B}_n]$ called Matsumoto section. For $0\leq k\leq n$, let $\mathfrak{S}_{k,n-k}\subset\mathfrak{S}_n$ denote the set of $(k,n-k)$-shuffles in $\mathfrak{S}_n$ defined by:
$$\mathfrak{S}_{k,n-k}=\{\sigma\in\mathfrak{S}_n|\ \sigma^{-1}(1)<\cdots<\s^{-1}(k),\ \s^{-1}(k+1)<\cdots<\s^{-1}(n)\}.$$

\begin{example} 
We explain the definition of these elements $\pp_{i,j}$ in an example for $\mathfrak{S}_4$:\\
$$\pp_{1,2}=\{(1),-(12)\},\ \ \pp_{1,3}=\pp_{2,3}\cup(\pp_{1,2}\circ -(132))=\{(1),-(23),-(132),(13)\},$$
$$\pp_{2,4}=\{(1),-(34),-(243),(24)\},$$
\begin{eqnarray*}
\pp_{1,4} &=& \pp_{2,4}\cup(\pp_{1,3}\circ -(1432))\\
&=& \{(1),-(34),-(243),(24),-(1432),(142),(1423),-(14)(23)\}.
\end{eqnarray*}
These elements $\mathcal{P}_{i,j}$ and $P_{i,j}$ come from iterated brackets: 
$$[a,[b,[c,d]]]=abcd-abdc-acdb+adcb-bcda+bdca+cdba-dcba.$$
When $\mathfrak{S}_4^\pm$ acts on letters $abcd$ by permuting their position and then multiplying by the signature, an easy computation gives:
$$[a,[b,[c,d]]]=T_{P_{1,4}}(abcd).$$
\end{example}
\begin{definition}
We call these $P_{i,j}$ Dynkin operators in $\mc[\mathfrak{S}_n]$ and corresponding elements $T_{P_{i,j}}$ Dynkin operators in $\mc[\mathfrak{B}_n]$.
\end{definition}
\indent
For $\sigma\in\mathfrak{S}_n$, let $l(\sigma)$ denote the length of $\sigma$. It is exactly the number of generators appearing in any reduced expression of $\sigma$.
\begin{remark}
In general, the Matsumoto section $T:\mathfrak{S}_n\ra \mathfrak{B}_n$ is not a group homomorphism, but we have the following property:
for $w,w'\in\mathfrak{S}_n$, if $l(ww')=l(w)+l(w')$, then $T_wT_{w'}=T_{ww'}$.
\end{remark}

\begin{lemma}\label{perm}
Let $w\in\mathcal{P}_{1,k}$ and $\sigma\in\mathfrak{S}_{k,n-k}$. Then $T_{w\sigma}=T_w T_\sigma$.
\end{lemma}
\begin{proof}
Recall that the length of an element in $\mathfrak{S}_n$ equals to the number of inversions of its action on $\{1,\cdots,n\}$. As $w$ permutes only the first $k$ positions, and the $(k,n-k)$-shuffle $\sigma$ preserves the order of first $k$ positions, the number of inversions of $w\sigma$ is the sum of those for $w$ and $\sigma$, which means that $l(w\sigma)=l(w)+l(\sigma)$ and then the lemma comes from the remark above.
\end{proof}

The following lemma is helpful for the understanding of the operator $P_{i,j}$ and for our further applications.

\begin{lemma}\label{decomposition:pn}
For $n\geq 2$ and $1\leq i<j\leq n$, the following identity holds in $\mc[\fb_n]$:
$$T_{P_{i,j}}=(1-\s_{j-1}\s_{j-2}\cdots\s_i)(1-\s_{j-1}\s_{j-2}\cdots\s_{i+1})\cdots(1-\s_{j-1}).$$
\end{lemma}
\begin{proof}
It suffices to show it for $i=1$ and $j=n$. We prove it by induction on $n$. The case $n=2$ is clear.
\par
Suppose that the lemma holds for $n-1$. From the definition of $P_{1,n}$, 
$P_{1,n}=P_{2,n}-P_{1,n-1}\circ(1,n,\cdots,2)$, so $P_{1,n}=P_{2,n}-(1,n,\cdots,2)\circ P_{2,n}$ and then
$$T_{P_{1,n}}=(1-\s_{n-1}\s_{n-2}\cdots\s_1)T_{P_{2,n}}.$$
From the induction hypothesis, 
$$T_{P_{2,n}}=(1-\s_{n-1}\s_{n-2}\cdots\s_2)\cdots(1-\s_{n-1}\s_{n-2})(1-\s_{n-1}),$$
which finishes the proof.
\end{proof}

\subsection{Properties of Dynkin operators}\label{3.2}
We treat $T(V)$ as a braided Hopf algebra as in Section \ref{section2}.
\par
At first, we define Dynkin operators on $T(V)$.
\begin{definition}
We define a graded endomorphism $
\Phi\in\bigoplus_{n=0}^\infty End(V^{\ts n})$ by: $\Phi(1)=0$ and for $x\in V^{\ts n}$ with $n\geq 1$,
$$\Phi(x)=T_{P_{1,n}}(x)\in V^{\ts n}.$$
It can be viewed as a linear map $\Phi:T(V)\ra T(V)$ and is called a Dynkin operator on $T(V)$.
\end{definition}

Using this notation, we can deduce from Lemma \ref{decomposition:pn} the following inductive characterization of $\Phi$:
for $v\in V$ and $w\in V^{\ts n}$, we have:
\begin{equation}\label{eqn1}
\Phi(vw)=\left\{\begin{array}{cc}vw,\ \ \  & \text{if}\ w\in\mathbb{C}, \\(1-T_{(1,n+1,\cdots,2)})(v\Phi(w)),\ \ \  & \text{if}\ n\geq 1.  \end{array}\right.
\end{equation}
Moreover, the following identity is clear:
\begin{equation}\label{eqn2}
T_{(1,n+1,\cdots,2)}(v\Phi(w))=(\Phi|_{V^{\ts n}}\ts id)(T_{(1,n+1,\cdots,2)}(vw)).
\end{equation}

The following proposition can be viewed as a generalization of a classical result for free Lie algebras in \cite{Reu93}. As $T(V)$ is a braided Hopf algebra, we let $\ast$ denote the convolution product in $End(T(V))$.
\begin{theorem}\label{theorem1}
Let $x\in V^{\ts n}$. Then
$$(\Phi\ast id)(x)=nx.$$
\end{theorem}
\begin{proof}
The proof is given by induction on the degree $n$. The case $n=1$ is trivial.
\par
Let $n\geq 2$. Suppose that the theorem holds for all elements of degree $n-1$. It suffices to show that for any $v\in V$ and $v\in V^{\ts n-1}$,
$$(\Phi\ast id)(vw)=nvw.$$
We write $\Delta(w)=1\ts w+\sum w'\ts w''$ where $w'\in\ker\ve=\bigoplus_{k=1}^\infty
V^{\ts k}$. For a homogeneous element $x\in T(V)$, we let $l(x)$ denote its degree. As $\Delta$ is an algebra morphism, with these notations,
$$\Delta(vw)=v\ts w+1\ts vw+\sum vw'\ts w''+(1\ts v)(\sum w'\ts w''),$$
and then
$$(\Phi\ast id)(vw)=\Phi(v)w+\sum\Phi(vw')w''+(\Phi|_{V^{\ts l(w')}}\ts id)(T_{(1,l(w')+1,\cdots,2)}(vw'))w''.$$
By the induction hypothesis, 
$$(n-1)w=(\Phi\ast id)(w)=\sum\Phi(w')w'',$$
then after (\ref{eqn1}),
$$\sum\Phi(vw')w''=\sum v\Phi(w')w''-\sum T_{(1,l(w')+1,\cdots,2)}(v\Phi(w'))w''.$$
In this formula, the first term is $(n-1)vw$ and the second one, after (\ref{eqn2}), equals to
$$(\Phi|_{V^{\ts l(w')}}\ts id)(T_{(1,l(w')+1,\cdots,2)}(vw'))w''.$$
Combining these formulas terminates the proof of the theorem.
\end{proof}
To write down the formula in a more compact form, we define the number operator:
\begin{definition}
The number operator $\mathcal{N}:T(V)\ra T(V)$ is the linear map defined by: $\mathcal{N}(1)=0$ and for any $x\in V^{\ts n}$ with $n\geq 1$,
$$\mathcal{N}(x)=nx.$$
\end{definition}
So the formula in Theorem \ref{theorem1} can be written as
$$(\Phi\ast id) (x)=\mathcal{N}(x).$$
As $S$ is the convolution inverse of the identity map, we have:
\begin{corollary}\label{NS}
Let $x\in T(V)$. Then:
$$(\mathcal{N}\ast S)(x)=\Phi(x).$$
\end{corollary}
As an application of Corollary \ref{NS}, we may descend $\Phi$ from braided tensor Hopf algebra $T(V)$ to Nichols algebra $\mathfrak{N}(V)$.
\begin{proposition}
$\Phi(\mathfrak{I}(V))\subset \mathfrak{I}(V)$, so $\Phi$ induces a linear map $\Phi:\mathfrak{N}(V)\ra\mathfrak{N}(V)$.
\end{proposition}
\begin{proof}
From the definition of $\mathfrak{I}(V)$, it is both a coideal and a two-sided ideal of $T(V)$. So the coproduct on it satisfies:
$$\Delta(\mathfrak{I}(V))\subset \mathfrak{I}(V)\ts T(V)+T(V)\ts\mathfrak{I}(V).$$
From Corollary \ref{NS}, $\Phi(\mathfrak{I}(V))=(\mathcal{N}\ast S)(\mathfrak{I}(V))\subset \mathfrak{I}(V)$ because $S(\mathfrak{I}(V))\subset \mathfrak{I}(V)$ and $\mathcal{N}$ respects $\mathfrak{I}(V)$ (note that $\mathfrak{I}(V)$ is a homogeneous ideal).
\end{proof}

\section{Decompositions in braid groups}
The objective of this section is to give a preparation for results that will be used in our investigations on the Dynkin operators and their relations with the structure of Nichols algebras. As such operators live in the group algebra $\mc[\fb_n]$, we would like to give first some decomposition results for some specific elements in $\mc[\fb_n]$.
\subsection{Central element}\label{4.1}
Let $n\geq 2$ be an integer and $Z(\fb_n)$ denote the center of $\fb_n$. 
\par
In the braid group $\mathfrak{B}_n$, there is a Garside element 
$$\Delta_n=(\sigma_1\cdots\sigma_{n-1})(\sigma_1\cdots\sigma_{n-2})\cdots(\sigma_1\sigma_2)\sigma_1.$$
The following characterization of $Z(\fb_n)$ is well known. 
\begin{proposition}[\cite{KT08}, Theorem 1.24]\label{cent}
For $n\geq 3$, let $\theta_n=\Delta_n^2$. Then $Z(\fb_n)$ is generated by $\theta_n$.
\end{proposition}
\indent
For the particular case $n=2$, we have $\theta_2=\Delta_2^2=\s_1^2$.
\par
Between lines of the proof of the proposition above in \cite{KT08}, the following lemma is obtained.
\begin{lemma}[\cite{KT08}] \label{changing}
For any $1\leq i\leq n-1$, $\sigma_i\Delta_n=\Delta_n\s_{n-i}$.
\end{lemma}

\begin{lemma}\label{moving}
The following identities hold:
\begin{enumerate}
\item
For any $1\leq s\leq n-2$, 
$$\s_s(\s_{n-1}\s_{n-2}\cdots\s_1)=(\s_{n-1}\s_{n-2}\cdots\s_1)\s_{s+1}.$$
\item
The element $\Delta_n$ has another presentation:
$$\Delta_n=\s_{n-1}(\s_{n-2}\s_{n-1})\cdots(\s_1\cdots\s_{n-2}\s_{n-1}).$$
\item
The element $\theta_n$ has another presentation:
$$\theta_n=\Delta_n^2=(\s_{n-1}\s_{n-2}\cdots\s_1)^n.$$
\end{enumerate}
\end{lemma}
\begin{proof}
\begin{enumerate}
\item This can be proved by a direct verification.
\item
$\Delta_n$ is the image under the Matsumoto section of the element $\sigma\in\mathfrak{S}_n$ such that for any $1\leq i\leq n$, $\sigma(i)=n-i+1$. It is easy to check that once projected to $\mathfrak{S}_n$, the element on the right hand side is exactly $\sigma$. Moreover, this decomposition is reduced because both sides have the same length, which finishes the proof.
\item
This identity comes from a direct computation using Lemma \ref{changing}:
\begin{eqnarray*}
\Delta_n^2 &=& \Delta_n\Delta_n\\
&=& \Delta_n (\s_1\s_2\cdots\s_{n-1})(\s_1\s_2\cdots\s_{n-2})\cdots(\s_1\s_2)\s_1\\
&=& (\s_{n-1}\s_{n-2}\cdots\s_1)(\s_{n-1}\s_{n-2}\cdots\s_2)\cdots(\s_{n-1}\s_{n-2})\s_{n-1}\Delta_n\\
&=& (\s_{n-1}\cdots\s_1)\cdots(\s_{n-1}\s_{n-2})\s_{n-1}(\s_1\cdots\s_{n-1})\cdots(\s_1\s_2)\s_1\\
&=& (\s_{n-1}\s_{n-2}\cdots\s_1)^n.
\end{eqnarray*}
\end{enumerate}
\end{proof}

The following proposition is the main result of this subsection.
\begin{proposition}
The element $\theta_n$ has another presentation:
$$\theta_n=\Delta_n^2=(\s_{n-1}^2\s_{n-2}\cdots\s_1)^{n-1}.$$
\end{proposition}
\begin{proof}
From Lemma \ref{moving},
$$\Delta_n^2=(\s_{n-1}\s_{n-2}\cdots\s_1)(\s_{n-1}\s_{n-2}\cdots\s_1)\cdots(\s_{n-1}\s_{n-2}\cdots\s_1).$$
At first, using Lemma \ref{moving}, it is possible to move the first $\s_1$ towards right until it can not move anymore. We exchange it with $(\s_{n-1}\s_{n-2}\cdots\s_1)$ for $n-2$ times, which gives $\s_{n-1}$ finally and so:
$$\Delta_n^2=(\s_{n-1}\s_{n-2}\cdots\s_2)(\s_{n-1}\s_{n-2}\cdots\s_1)\cdots(\s_{n-1}^2\s_{n-2}\cdots\s_1).$$
Repeating this procedure with the help of Lemma \ref{moving} for the first $\s_2,\cdots,\s_{n-2}$, we will obtain the presentation as announced in the proposition.
\end{proof}

\subsection{Decompositions in the group algebra}\label{decomposition}
In this subsection, we will work in the group algebra $\mc[\fb_n]$ for some $n\geq 2$.
\par
The symmetrization operator in $\mathbb{C}[\mathfrak{B}_n]$ is defined by:
$$S_n=\sum_{\sigma\in\mathfrak{S}_n}T_\sigma\in\mc[\fb_n].$$
Because $V$ is a braided vector space and $\fb_n$ acts naturally on $V^{\ts n}$, we may treat $S_n$ as a linear operator in $End(V^{\ts n})$.
\par
For $1\leq i\leq n-1$, let $(i,i+1)\in\mathfrak{S}_n$ be a transposition. We have seen that $T_{(i,i+1)}=\sigma_i$.
\par
The element $S_n\in\mc[\fb_n]$ has a remarkable decomposition as shown in \cite{DKKT}. For any $2\leq m\leq n$, we define 
$$T_m=1+\s_{m-1}+\s_{m-1}\s_{m-2}+\cdots+\s_{m-1}\s_{m-2}\cdots\s_1\in\mc[\fb_n].$$
\begin{proposition}[\cite{DKKT}]\label{decomposition:sn}
For any $n\geq 2$, 
$$S_n=T_2T_3\cdots T_n\in\mc[\fb_n].$$
\end{proposition}
In fact, this proposition is true when being projected to $\mc[\mathfrak{S}_n]$; then notice that the expansion of the product on the right hand side contains only reduced terms.
\par
Recall the definition of $P_{1,n}$ in Section \ref{Dynkin}. To simplify the notation, we denote 
$$P_n=T_{P_{1,n}}\in\mc[\fb_n].$$
This element will be an important ingredient in our further discussion.
\par
For $n\geq 2$, recall the decomposition of $P_n$ given in Lemma \ref{decomposition:pn}:
$$P_{n}=(1-\s_{n-1}\s_{n-2}\cdots\s_1)(1-\s_{n-1}\s_{n-2}\cdots\s_{2})\cdots(1-\s_{n-1}).$$

This element $P_n$ permits us to give a much more refined structure of $T_n$. We introduce another member
$$T_n'=(1-\s_{n-1}^2\s_{n-2}\cdots\s_1)(1-\s_{n-1}^2\s_{n-2}\cdots\s_2)\cdots(1-\s_{n-1}^2)\in\mathbb{C}[\mathfrak{B}_n].$$
\begin{proposition}\label{decomposition:tn}
For $n\geq 2$, the decomposition $T_nP_n=T_n'$ holds in $\mc[\fb_n]$.
\end{proposition}
\begin{proof}
The Proposition 6.11 in \cite{DKKT} affirms that if all inverses appearing are well defined, then
$$T_n=(1-\s_{n-1}^2\s_{n-2}\cdots\s_1)\cdots(1-\s_{n-1}^2)(1-\s_{n-1})^{-1}\cdots(1-\s_{n-1}\cdots\s_1)^{-1}.$$
So the proposition follows from Lemma \ref{decomposition:pn}.
\end{proof}
\begin{corollary}\label{1-thetan}
The following identity holds in $\mc[\fb_n]$:
$$\left(\sum_{k=0}^{n-2}(\s_{n-1}^2\s_{n-2}\cdots\s_1)^k\right)(1-\s_{n-1}^2\s_{n-2}\cdots\s_1)=1-\Delta_n^2=1-\theta_n.$$
Moreover, for $3\leq s\leq n-1$, let $\iota_s:\fb_s\hookrightarrow\fb_n$ be the canonical embedding of braid groups on the last $s$ strands. If $\theta_s$ is the central element in $\fb_s$, we denote $\theta_s^{\iota_s}=\iota_s(\theta_s)$ and $\theta_2^{\iota_2}=\s_{n-1}^2$, then there exists an element
$$L_n=\left(\sum_{k=0}^1(\s_{n-1}^2\s_{n-2})^k\right)\cdots\left(\sum_{k=0}^{n-2}(\s_{n-1}^2\s_{n-2}\cdots\s_1)^k\right)\in\mc[\fb_n],$$
such that in $\mc[\fb_n]$,
$$L_nT_n'=(1-\theta_n)(1-\theta_{n-1}^{\iota_{n-1}})\cdots (1-\theta_2^{\iota_2}).$$
\end{corollary}

\section{The study of the ideal $\mathfrak{I}(V)$}
We keep assumptions and notations in previous sections.
\subsection{A result on the quotient ideal}
As we have seen, the Nichols algebra associated to an $H$-Yetter-Drinfel'd module $V$ is a quotient of the braided tensor Hopf algebra $T(V)$ by a maximal coideal $\mathfrak{I}(V)$ contained in $T^{\geq 2}(V)$. This definition tells us almost nothing about the concrete structure of $\mathfrak{I}(V)$: as $\mathfrak{I}(V)$ is an ideal, the Nichols algebra $\mathfrak{N}(V)$ can be viewed as imposing some relations in $T(V)$, but such relations can never be read directly from the definition.
\par
As we know, the best result for the structure of $\mathfrak{I}(V)$ is obtained by M. Rosso in \cite{Ros98} in a dual point of view and by P. Schaurenburg in \cite{Sbg96}. We recall this result in this subsection.
\par
Let $S_n:V^{\ts n}\ra V^{\ts n}$ be the element in $\mc[\fb_n]$ defined in Section \ref{decomposition}.
\begin{proposition}[\cite{Ros98}, \cite{Sbg96}]\label{schauenburg}
Let $V$ be an $H$-Yetter-Drinfel'd module. Then
$$\mathfrak{N}(V)=\bigoplus_{n\geq 0} \left(V^{\ts n}/\ker(S_n)\right).$$
\end{proposition}
So to make more precise the structure of $\mathfrak{I}(V)$, it suffices to study each subspace $\ker(S_n)$. In the following part of this section, we want to characterize a part of elements in $\ker(S_n)$ and show that in cases of great interest, this part is the essential one for understanding the structure of $\ker(S_n)$.

\subsection{General assumption}\label{assumption}
From now on, assume that $n\geq 2$ is an integer. To study the structure of $\ker(S_n)$, we want first to concentrate on some essential "levels" in it.
\begin{definition}
Let $1<s<n$ be an integer and $i:\mathfrak{B}_s\ra \mathfrak{B}_n$ be an injection of groups. We call $i$ a positional embedding if there exists some integer $0\leq r\leq n-s$ such that for any $1\leq t\leq s-1$, $i(\s_t)=\s_{t+r}$.  
\end{definition}

\indent
For an element $v\in V^{\ts n}$, if $v\in\ker(S_n)$, there are two possibilities:
\begin{enumerate}
\item There exists some $2\leq s<n$ and some positional embedding of groups $\iota:\mathfrak{B}_s\hookrightarrow\mathfrak{B}_n$ such that $v$ is annihilated by $\iota(S_s)$;
\item For any $s$ and positional embedding $\iota$ as above, $v$ is not in $\ker(\iota(S_s))$.
\end{enumerate}
\indent
Elements falling in the case (2) are much more interesting in our framework. So we would like to give a more concrete assumption for the purpose of concentrating on such elements; here, we want to impose a somehow stronger restriction.
\par
Let $v\in V^{\ts n}$ be a non-zero element and $\mc[X_v]$ denote the $\mc[\fb_n]$-submodule of $V^{\ts n}$ generated by $v$, that is to say, $\mc[X_v]=\mc[\fb_n].v$. Because $\mc[X_v]$ is a $\mc[\fb_n]$-module, $S_n:\mc[X_v]\ra\mc[X_v]$ is well defined.
\par
We fix this $v\in V^{\ts n}$ as above, the restriction on $v$ we want to impose is as follows:
\begin{definition}\label{Def:leveln}
An element $v\in V^{\ts n}$ is called of level $n$ if $S_nv=0$ and for any $2\leq s\leq n-1$ and any positional embedding $\iota:\fb_s\hookrightarrow\fb_n$, the equation $\iota(\theta_s)x=x$ has no solution in $\mc[X_v]$.
\end{definition}
\begin{proposition}\label{prop1}
If $v\in V^{\ts n}$ is a non-zero solution of equation $(1-\s_{n-1}^2\s_{n-2}\cdots\s_1)x=0$, then $\theta_n.v=v$.
\end{proposition}
\begin{proof}
If $(1-\s_{n-1}^2\s_{n-2}\cdots\s_1)v=0$, from Corollary \ref{1-thetan}, 
$$0=\left(\sum_{k=0}^{n-2}(\s_{n-1}^2\s_{n-2}\cdots\s_1)^k\right)(1-\s_{n-1}^2\s_{n-2}\cdots\s_1)v=(1-\theta_n)v.$$
\end{proof}
\begin{remark}\label{rem}
As $V^{\ts n}$ is a $\mc[\fb_n]$-module, we can define $\mathcal{H}\subset V^{\ts n}$ as the subspace of $V^{\ts n}$ formed by eigenvectors of $\theta_n$ with eigenvalue $1$ (that is to say, $\mathcal{H}=\{w\in V^{\ts n}|\ \theta_nw=w\}$).
\par
As $\theta_n\in Z(\fb_n)$, for any $w\in \mathcal{H}$ and $Y\in\mc[\fb_n]$, we have $\theta_nYw=Yw$, thus $\mathcal{H}$ is a $\mc[\fb_n]$-submodule of $V^{\ts n}$. It means that if $v\in \mathcal{H}$, then $\mc[X_v]\subset \mathcal{H}$.
\end{remark}
\begin{lemma}\label{2->1}
Let $v\in V^{\ts n}$ be a non-zero element of level $n$. Then for any $2\leq i\leq n-1$ and any positional embedding $\iota_i:\mathfrak{B}_i\hookrightarrow\mathfrak{B}_n$, $\iota_i(S_i)x=0$ has no solution on $\mc[X_v]$.
\end{lemma}
\begin{proof}

Let $\iota_i:\mathfrak{B}_i\hookrightarrow\mathfrak{B}_n$ be a positional embedding such that $\iota_i(S_i)x=0$ has a solution in $\mc[X_v]$. Then $\iota_i(S_i)=\iota_i(T_2)\cdots \iota_i(T_i)\in\mc[\fb_n]$. The equation $\iota_i(S_i)x=0$ has a solution in $\mc[X_v]$ means that $det(\iota_i(S_i))=0$, so there exists some $2\leq j\leq i$ such that $det(\iota_i(T_j))=0$.
\par
Because $\iota_i(T_j)\iota_i(P_j)=\iota_i(T_j')$, we obtain that $det(\iota_i(T_j'))=0$. From the definition of $T_j'$, there exists some $1\leq k\leq j-1$ such that
$$det(1-\iota_i(\s_{j-1}^2\s_{j-2}\cdots\s_k))=0.$$
So we can choose another positional embedding $\iota:\fb_{j-k+1}\ra\fb_n$ such that for the action of $\mc[\fb_n]$ on $\mc[X_v]$, 
$$det(1-\iota(\s_{j-k}^2\s_{j-k-1}\cdots\s_1))=0.$$
So, from Proposition \ref{prop1}, $\iota(\theta_{j-k+1})x=x$ has a non-zero solution on $\mc[X_v]$, which contradicts to the assumption that $v$ is of level $n$.
\end{proof}

\subsection{Solutions}\label{solutions}
Fix some $n\geq 2$, we want to solve the equation $S_nx=0$ on $\mc[X_v]$ for some non-zero element $v\in V^{\ts n}$.
\par
We define an element in $\mathbb{C}[\mathfrak{B}_n]$:
$$X=(1-\s_{n-1}^2\s_{n-2}\cdots\s_2)\cdots(1-\s_{n-1}^2\s_{n-2}\cdots\s_3)\cdots(1-\s_{n-1}^2\s_{n-2})(1-\s_{n-1}^2).$$
\begin{proposition}\label{prop2}
If $v\in V^{\ts n}$ is a non-zero element of level $n$, then $X$ is invertible on $\mc[X_v]$.
\end{proposition}
\begin{proof}
We may view $X$ as an element in $End(\mc[X_v])$. If $X$ is not invertible, $det(X)=0$. From the definition, there must exist some term, say $(1-\s_{n-1}^2\s_{n-2}\cdots\s_i)$, for some $2\leq i\leq n-1$, having determinant $0$. So there exists some nonzero element $w\in\mc[X_v]$ such that $(1-\s_{n-1}^2\s_{n-2}\cdots\s_i)w=0$. But from Proposition \ref{prop1}, we may find some positional embedding $\iota:\fb_{n-i+1}\hookrightarrow\fb_n$ such that $\iota(\theta_{n-i+1})w=w$, which contradicts to the assumption that $v$ is of level $n$.
\end{proof}
The level $n$ assumption we are working with will give more information on solutions of equation $S_nx=0$.

\begin{proposition}\label{prop3}
Let $v\in V^{\ts n}$ be a non-zero element of level $n$.
\begin{enumerate}
\item
There exists a bijection between nonzero solutions of equation $T_n'x=0$ and of the equation $(1-\s_{n-1}^2\s_{n-2}\cdots\s_1)x=0$ in $\mc[X_v]$.
\item
Equations $S_nx=0$ and $T_nx=0$ have the same solutions in $\mc[X_v]$.
\end{enumerate}
\end{proposition}
\begin{proof}
\begin{enumerate}
\item
From Proposition \ref{prop2}, $X^{-1}:\mc[X_v]\ra\mc[X_v]$ is well defined. So this proposition comes from the identity: $T_n'=(1-\s_{n-1}^2\s_{n-2}\cdots\s_1)X$.
\item
Let $w$ be a non-zero solution of $T_nx=0$, then from Proposition \ref{decomposition:sn}, $S_nw=0$.
\par
Conversely, let $u$ be a non-zero solution of $S_nx=0$. If $T_nu\neq 0$, again from Proposition \ref{decomposition:sn}, $T_nu$ will be a non-zero solution of equation $T_2\cdots T_{n-1}x=0$ on $\mc[X_v]$, so $S_{n-1}x=0$ has a non-zero solution on $\mc[X_v]$, contradicts Lemma \ref{2->1} above.
\end{enumerate}
\end{proof}

\indent
Recall that $P_n=T_{P_{1,n}}\in\mc[\fb_n]$ as defined in the last section, $P_n\in End(\mc[X_v])$.
\par 
Now, let $w\in\ker(S_n)\cap Im(P_n)$ be a non-zero element of level $n$. Then from Lemma \ref{2->1}, $w$ satisfies $T_nw=0$. Moreover, because it is in $Im(P_n)$, we can choose some $w'$ such that $P_n(w')=w$, then 
$$T_n'w'=T_nw=0.$$
From the identity $T_n'=(1-\s_{n-1}^2\s_{n-2}\cdots\s_1)X$, $Xw'$ is a solution of the equation $(1-\s_{n-1}^2\s_{n-2}\cdots\s_1)x=0$, so from Proposition \ref{prop1}, $\theta_nXw'=Xw'$. This discussion gives the following proposition.
\begin{proposition}
Let $w\in \ker(S_n)\cap Im(P_n)$ be an element of level $n$. Then $\theta_nw=w$.
\end{proposition}
\begin{proof}
From the definition of $w'$, if $\theta_n$ fixes $w'$, then it fixes $w$. So if $\theta_nw\neq w$, then it does not fix $w'$ and then $Xw'$ (see Remark \ref{rem}), which is a contradiction.
\end{proof}
\begin{remark}\label{remark}
Let $\mathcal{H}$ denote the eigenspace of $\theta_n$ corresponding to the eigenvalue $1$ as in the Remark \ref{rem} above. If we let $E_n$ denote the set of elements in $Im(P_n)$ with level $n$ in $V^{\ts n}$, then the proposition above implies that $E_n\subset P_n(\mathcal{H})$.
\end{remark}
\indent
We have constructed solutions of equation $\theta_nx=x$ on $\mc[X_v]$ from some kinds of elements in $\ker(S_n)\cap Im(P_n)$. Now we proceed to consider the construction in the opposite direction.
\par
Let $w\in\mc[X_v]$ be a solution of $\theta_nx=x$. If 
$$u=\left(\sum_{k=0}^{n-2}(\s_{n-1}^2\s_{n-2}\cdots\s_1)^k\right)w\neq 0,$$ 
it will be a solution of the equation $(1-\s_{n-1}^2\s_{n-2}\cdots\s_1)x=0$, then $X^{-1}u$ is a solution of the equation $T_n'x=0$ and $P_nX^{-1}u$ will be a non-trivial solution of $S_nx=0$ if it is not zero; moreover, it is in $Im(P_n)$, from which we obtain an element in $\ker(S_n)\cap Im(P_n)$.
\par
There are some possibilities for the appearance of zero elements when passing from the solutions of $\theta_nx=x$ to those of $S_nx=0$. The appearance of zeros mostly comes from the fact that an element satisfying $\theta_nx=x$ may be the solution of $\iota_s(\theta_s)x=x$ for some $2\leq s\leq n-1$ and some positional embedding $\iota_s:\mathfrak{B}_s\hookrightarrow\mathfrak{B}_n$. 
\par
The subspace $\ker(S_n)\cap Im(P_n)$ is sufficiently important, as will be shown in the next subsection.

\subsection{Properties of $\ker(S_n)\cap Im(P_n)$}
In this subsection, suppose that $n\geq 2$ is an integer.
\par
Instead of $End(\mc[X_v])$, it is better in this subsection to view $S_n,P_n$ as elements in $End(V^{\ts n})$. We want to show that $\ker(S_n)\cap Im(P_n)$ contains all primitive elements and in some special cases (for example, the diagonal case), it generates $\ker(S_n)$.
\begin{proposition}\label{primitive}
Let $v\in V^{\ts n}$ be a homogeneous primitive element of degree $n$. Then $v\in\ker(S_n)\cap Im(P_n)$.
\end{proposition}
\begin{proof}
The fact $v\in\ker(S_n)$ is a corollary of the definition of Nichols algebra and Proposition \ref{schauenburg}. So it suffices to show that $v\in Im(P_n)$.
\par
The element $v$ is primitive means that $\Delta(v)=v\ts 1+1\ts v$. From Theorem \ref{theorem1}, $\Phi\ast id=\mathcal{N}$, so
$$nv=\Phi\ast id(v)=m\circ(\Phi\ts id)\Delta(v)=\Phi(v),$$
and then
$$v=\frac{1}{n}\Phi(v)=\frac{1}{n}P_n(v)\in Im(P_n).$$
\end{proof}
\indent
The second property we want to establish is that in the diagonal case, these subspaces $\ker(S_n)\cap Im(P_n)$ will generate the ideal (also coideal) $\mathfrak{I}(V)$.
\par
Recall that from the definition of Nichols algebra and Proposition \ref{schauenburg}, the subspace 
$$\mathfrak{I}(V)=\bigoplus_{n\geq 2}\ker(S_n)\subset T(V)$$
is a maximal coideal contained in $T^{\geq 2}(V)$. Moreover, it is a homogeneous ideal.
\par
Let $J\subset T^{\geq 2}(V)$ be a coideal in ${}^H_H\mathcal{YD}$ containing the subspace
$$\bigoplus_{n\geq 2}\left(\ker(S_n)\cap Im(P_n)\right).$$
Such a coideal does exist as $\mathfrak{I}(V)$ satisfies these conditions.

\begin{proposition}\label{generate}
Let $T(V)$ be of diagonal type. Then the ideal generated by $J$ in $T(V)$ is $\mathfrak{I}(V)$.
\end{proposition}

\begin{proof}
Let $K$ be the two-sided ideal generated by $J\subset T^{\geq 2}(V)$ in $T(V)$. Then $K$ is also an ideal in $T^{\geq 2}(V)$. As a two-sided ideal generated by a coideal, $K$ is also a coideal.  
From the maximality of $\mathfrak{I}(V)$, $K\subset\mathfrak{I}(V)$.
\par
We proceed to prove that $T(V)/K\cong\mathfrak{N}(V)$. For this purpose, the following lemma is needed.
\begin{lemma}[\cite{HecNote}]
Suppose that the Nichols algebra is of diagonal type. 
Let $K\subset T^{\geq 2}(V)$ be simultaneously an ideal, a coideal and an $H$-Yetter-Drinfel'd module. If all primitive elements in $T(V)/K$ are concentrated in $V$, then $T(V)/K\cong\mathfrak{N}(V)$.
\end{lemma}

From this lemma, it suffices to show that there is no non-zero primitive element of degree greater than $1$ in $T(V)/K$.
\par
Suppose that $v$ is such a non-zero element which is moreover homogeneous of degree $n$, so in $T(V)$, 
$$\Delta(v)\in v\ts 1+1\ts v+K\ts T(V)+T(V)\ts K.$$
As $K\subset \mathfrak{I}(V)$, 
$$\Delta(v)\in v\ts 1+1\ts v+\mathfrak{I}(V)\ts T(V)+T(V)\ts\mathfrak{I}(V).$$
But in $T(V)/\mathfrak{I}(V)$, from the definition of Nichols algebra, there is no such element, which forces $v\in\mathfrak{I}(V)$ and then $S_nv=0$.
\par
We need to show that in fact $v\in K$. From Corollary \ref{NS}, 
$$P_n(v)=\mathcal{N}\ast S(v)\in nv+K,$$
then $v-k\in Im(P_n)$ for some $k\in K$. As $S_nv=0$ and $K\subset\ker(S_n)$, $v-k\in\ker(S_n)\cap Im(P_n)\subset K$; this implies $v\in K$.
\end{proof}

\indent
This proposition shows the importance of these subspaces $\ker(S_n)\cap Im(P_n)$ in the study of the defining ideal.

\subsection{Main theorem}
The main result of this paper is:
\begin{theorem}\label{maintheorem}
Elements of level $n$ are primitive.
\end{theorem}
\indent
From this theorem, level $n$ solutions of $S_nx=0$ are primitive elements of degree $n$, so they are in $\ker(S_n)\cap Im(P_n)$. Moreover, this introduces a method to find primitive elements in $T(V)$.
\par
The proof of this theorem will be given in the end of this paper, after introducing the differential algebra of a Nichols algebra.

\section{Applications}
In this section, we give some applications of the machinery constructed above.
\par
Though the discussion in the last section is somehow elementary, it may give remarkable results and good points of view once being applied to some concrete examples.

\subsection{A general application for the diagonal type}\label{diagonaltype}
Let $H$ be the group algebra of an abelian group $G$, $V\in{}_H^H\mathcal{YD}$ be of diagonal type, $T(V)$ and $\mathfrak{N}(V)$ be the braided tensor Hopf algebra and Nichols algebra, respectively.
\par
Suppose that $dimV=m$, with basis $v_1,\cdots,v_m$ such that 
$$\sigma(v_i\ts v_j)=q_{ij}v_j\ts v_i.$$
From the definition of the braiding, the action of $\mc[\fb_n]$ on $V^{\ts n}$ has the following decomposition
$$V^{\ts n}=\bigoplus_{\underline{i}\in I}\mc[\fb_n].v_1^{i_1}\cdots v_m^{i_m},$$
where the sum runs over $I=\{\underline{i}=(i_1,\cdots,i_m)|\ i_1+\cdots+i_m=n\}$.
\par
We fix some $\underline{i}=(i_1,\cdots,i_m)\in\mathbb{N}^m$ such that $\underline{i}\in I$ and a monomial $v_{\underline{i}}=v_1^{i_1}\cdots v_m^{i_m}$, then in 
$$\mc[X_{\underline{i}}]=\mc[\fb_n].v_{\underline{i}},$$
if $\theta_nx=x$ has a solution, we must obtain $\theta_n v_{\underline{i}}=v_{\underline{i}}$.
\par
Indeed, when projected canonically to $\mathfrak{S}_n$, the element $\theta_n\in\fb_n$ corresponds to $1$, so if $\theta_nx=x$, $\theta_n$ will stablize all  components of $x$. From the decomposition above, for any component $x_0$ of $x$, there exists a nonzero constant $c$ and an element $\sigma\in\fb_n$ such that $cx_0=\s(v_{\underline{i}})$, thus $v_{\underline{i}}=\s^{-1}(cx_0)$. Thus all level $n$ elements are contained in the sum of some $\mc[X_{\underline{i}}]$ for some $v_{\underline{i}}$ satisfying $\theta_nv_{\underline{i}}=v_{\underline{i}}$.
\par
To exclude those elements which have not level $n$ but are stable under the action of $\theta_n$, some notations are needed.
\par
We fix some $\underline{i}$ and $v_{\underline{i}}$. It is more convenient to write $v_{\underline{i}}=e_1\cdots e_n$, where $e_i$ are some $v_j$'s. Then let $T_{\underline{i}}=(t_{ij})$ denote a matrix in $M_n(\mc)$ with $t_{ii}=1$ and for $i\neq j$, $t_{ij}$ are defined by 
$$\sigma(e_i\ts e_j)=t_{ij}e_j\ts e_i.$$
For some $2\leq s\leq n$, $1\leq k_1<\cdots<k_s\leq n$, $\underline{k}=(k_1,\cdots,k_s)$, we define:
$$\Pi_s^{\underline{k}}=\prod_{i=1}^s\prod_{j=1}^s t_{k_i,k_j}.$$
\indent
The following proposition is an easy consequence of the definition.
\begin{proposition} 
With the notations above, we have:
\begin{enumerate}
\item
$\theta_n v_{\underline{i}}=v_{\underline{i}}$ if and only if $\Pi_n^{\underline{k}}=1$.
\item
If for any $2\leq s\leq n-1$ and any $\underline{k}$, $\Pi_s^{\underline{k}}\neq 1$, then $v_{\underline{i}}$ satisfies the assumption in Definition \ref{Def:leveln}. Moreover, all elements in $\ker(S_n)\cap\mc[X_{v_{\underline{i}}}]$ are of level $n$.
\end{enumerate}
\end{proposition}

\begin{remark}\label{remarkdiag}
Under the assumptions (2) in the proposition above,
\begin{enumerate}
\item From Theorem \ref{maintheorem}, all elements in $\ker(S_n)\cap\mc[X_{v_{\underline{i}}}]$ are primitive.
\item From Remark \ref{remark}, all elements in $\ker(S_n)\cap\mc[X_{v_{\underline{i}}}]$ can be constructed from $\mc[X_{v_{\underline{i}}}]$ by the method given in the end of Section \ref{solutions}. So such a family of primitive elements can be easily and directly computed.
\end{enumerate}
\end{remark}

\begin{remark}\label{level2}
If $v\in V\ts V$ is an element in $\ker(S_2)\cap Im(P_2)$, then it must be of level $2$ and so primitive. Moreover, in the diagonal case, level $2$ elements in $\ker(S_2)\cap Im(P_2)$ can be obtained from monomials stablized by $\theta_2$ by applying $P_2$.
\end{remark}

\subsection{Exterior algebras}
In this subsection, as a warm up, we apply results of the previous section to the construction of exterior algebras.
\par 
The main ingredient is the Hopf algebra $H=\mc[G]$, where $G=\mathbb{Z}/2\mathbb{Z}=\{1,\ve\}$. Let $V$ be a finite dimensional vector space with basis $v_1,\cdots,v_m$. 
\begin{enumerate}
\item The action of $H$ on $V$ is given by: for $v\in V$, $\ve.v=-v$; 
\item The coaction is given by: $\delta(v)=\ve\ts v$, where $\delta:V\ra H\ts V$. 
\end{enumerate}
This makes $V$ an $H$-Yetter-Drinfel'd module.
\par
We form the braided Hopf algebra $T(V)$ and want to calculate relations appearing in the ideal $\mathfrak{I}(V)$.
\par
At first, we consider relations in $V^{\ts n}$ of level $n$. In fact, for $n\geq 3$, there are no such relations because if $v=v_{i_1}\cdots v_{i_n}\in V^{\ts n}$ is a pure tensor such that $\theta_nv=v$, from the definition of the braiding, there must exist 
some $1\leq s<t\leq n$ such that $\s^2(v_{i_s}\ts v_{i_t})=v_{i_s}\ts v_{i_t}$, which contradicts the definition of level $n$ relations.
\par
So it suffices to consider relations of level $2$ in $V^{\ts 2}$. We start from considering all solutions of $\theta_2x=x$ in $V^{\ts 2}$. These solutions are:
$v_iv_j$, for $1\leq i,j\leq n$.
\par
As in the procedure of constructing solutions of $S_nx=0$ from $\theta_nx=x$ given in the last section, the action of $P_2$ on these elements gives: 
$$P_2(v_iv_j)=v_iv_j+v_jv_i,$$
so $v_iv_j+v_jv_i\in\ker(S_2)\cap Im(P_2)$. Moreover, from Remark \ref{level2} in the last subsection, we obtain
$$\ker(S_2)\cap Im(P_2)=span\{v_iv_j+v_jv_i|\ 1\leq i,j\leq n\}.$$

\subsection{Quantized enveloping algebras}\label{qea}
In this subsection, we will discover the quantized Serre relations in the definition of the quantized enveloping algebra $\uqg$ associated to a symmetrizable Kac-Moody Lie algebra $\g$ by assuming almost no knowledge about the existence of such relations.
\par
Let $q$ be a nonzero complex number such that for any $N\geq 1$, $q^N\neq 1$. Let $\g$ be a symmetrizable Kac-Moody Lie algebra of rank $n$, $C=(C_{ij})_{n\times n}$ be its generalized Cartan matrix and $A=DC$ be the symmetrization of the Cartan matrix by some diagonal matrix $D=(d_1,\cdots,d_n)$ with $d_i$ positive integers which are relatively prime. We denote $A=(a_{ij})_{n\times n}$.
\par
At first, we briefly recall the construction of the strict positive part of $\uqg$ in the framework of Nichols algebras. This construction is due to M. Rosso and can be found in \cite{Ros98} with a slightly different language.
\par
Let $H=\mc[G]$ be the group algebra where $G$ is the abelian group $\mathbb{Z}^n$. Let $K_1,\cdots,K_n$ denote a basis of $\mathbb{Z}^n$. Then $H$ is a commutative and cocommutative Hopf algebra.
\par
Let $V$ be a $\mc$-vector space of dimension $n$ with basis $E_1,\cdots,E_n$. We define an $H$-Yetter-Drinfel'd module structure on $V$ by: 
\begin{enumerate}
\item The action of $K_i$ on $E_j$ is given by: $K_i.E_j=q^{a_{ij}}E_j$;
\item The coaction of $E_i$ is given by: $\delta(E_i)=K_i\ts E_i$, where $\delta:V\ra H\ts V$ is the structure map of left $H$-comodule structure on $V$.
\end{enumerate}
\indent
Starting with this $V\in{}_H^H\mathcal{YD}$, the braided tensor algebra $T(V)$ and the corresponding Nichols algebra $\mathfrak{N}(V)$ can be constructed. The defining ideal is denoted by $\mathfrak{I}(V)$.
\par
Assume that we know nothing about this ideal $\mathfrak{I}(V)$ (because from the general theory of quantized enveloping algebras, we know that $\mathfrak{I}(V)$ is generated by quantized Serre relations). Here, our point of view is much more pedestrian: if we do not know them, how to find?
\par
Results in the previous section will offer us a method.
\par
At first, we want to concentrate on the case $\mathcal{U}_q(\mathfrak{sl}_3)$, the simplest one which has such quantized Serre relations. In this case, $H=\mathbb{Z}^2$ with basis $K_1,K_2$, $V$ is of dimension $2$ with basis $E_1,E_2$.\\ 
\indent
We would like to compute the level $3$ relations in $\mathfrak{I}(V)$.
\par
At first, we write down all monomials of degree $3$ which are stabilized by the action of $\theta_3$ but not for all $\theta_2$ with possible embeddings. They are:
$$E_1^2E_2,\ \ E_1E_2E_1,\ \ E_2E_1^2,\ \ E_1E_2^2,\ \ E_2E_1E_2,\ \ E_2^2E_1.$$
After the action of $1+\s_2^2\s_1$, we obtain:
$$2E_1^2E_2,\ \ E_1E_2E_1+q^3E_2E_1^2,\ \ E_2E_1^2+q^{-3}E_1E_2E_1,$$
$$E_1E_2^2+q^{-3}E_2E_1E_2,\ \ E_2E_1E_2+q^3E_1E_2^2,\ \ 2E_2^2E_1.$$
In this case, $X=1-\s_2^2$, so the action of $X^{-1}$ on these elements will give:
$$x_1=\frac{2}{1-q^{-2}}E_1^2E_2,\ \ x_2=\frac{1}{1-q^{-2}}E_1E_2E_1+\frac{q^3}{1-q^4}E_2E_1^2,$$
$$x_3=\frac{1}{1-q^4}E_2E_1^2+\frac{q^{-3}}{1-q^{-2}}E_1E_2E_1,\ \ x_4=\frac{1}{1-q^4}E_1E_2^2+\frac{q^{-3}}{1-q^{-2}}E_2E_1E_2,$$
$$x_5=\frac{1}{1-q^{-2}}E_2E_1E_2+\frac{q^3}{1-q^4}E_1E_2^2,\ \ x_6=\frac{1}{1-q^{-2}}E_2^2E_1.$$
It is easy to compute the action of $P_3$ on all possible monomials:
$$P_3(E_1^2E_2)=E_1^2E_2-(q+q^{-1})E_1E_2E_1+E_2E_1^2,$$
$$P_3(E_1E_2E_1)=2E_1E_2E_1-q^{-1}E_1^2E_2-qE_2E_1^2,$$
$$P_3(E_2E_1^2)=(1-q^2)E_2E_1^2-(q^{-2}-1)E_1^2E_2,$$
$$P_3(E_1E_2^2)=(1-q^2)E_1E_2^2-(q^{-2}-1)E_2^2E_1,$$
$$P_3(E_2E_1E_2)=2E_2E_1E_2-qE_1E_2^2-q^{-1}E_2^2E_1,$$
$$P_3(E_2^2E_1)=E_2^2E_1-(q+q^{-1})E_2E_1E_2+E_1E_2^2.$$
And then
$$P_3(x_1)=\frac{2}{1-q^{-2}}(E_1^2E_2-(q+q^{-1})E_1E_2E_1+E_2E_1^2),$$
$$P_3(x_2)=-\frac{2q^{-1}}{1-q^{-4}}(E_1^2E_2-(q+q^{-1})E_1E_2E_1+E_2E_1^2),$$
$$P_3(x_3)=\frac{2}{1-q^4}(E_1^2E_2-(q+q^{-1})E_1E_2E_1+E_2E_1^2),$$
$$P_3(x_4)=\frac{2}{1-q^4}(E_2^2E_1-(q+q^{-1})E_2E_1E_2+E_1E_2^2),$$
$$P_3(x_5)=-\frac{2q^{-1}}{1-q^{-4}}(E_2^2E_1-(q+q^{-1})E_2E_1E_2+E_1E_2^2),$$
$$P_3(x_6)=\frac{2}{1-q^{-2}}(E_2^2E_1-(q+q^{-1})E_2E_1E_2+E_1E_2^2).$$
\indent
So starting with solutions of $\theta_3x=x$ with level $3$, the solutions in $Im(P_3)$ with level $3$ we obtained for the equation $S_3x=0$ are exactly the quantized Serre relations of degree $3$.
\par
Moreover, we show that there are no other relations of level $3$. If $w\in\ker(S_3)$ is such an element, it will be stable under the action of $\theta_3$, so it is a linear combination of monomials above, then it must be a linear combination of degree $3$ Serre relations.
\par
Finally, we turn to the level $n$ elements for an arbitrary integer $n\geq 2$. As explained in Section \ref{diagonaltype}, it suffices to consider a monomial of form $E_1^sE_2^t$ for some positive integers $s$ and $t$.
\par
The action of $\theta_{s+t}$ on this monomial gives:
$$\theta_{s+t}(E_1^sE_2^t)=q^{s^2-s+t^2-t-st}E_1^sE_2^t.$$
So this monomial is stablized by $\theta_{s+t}$ if and only if
$$s^2-s+t^2-t-st=\frac{1}{2}\left((s-t)^2+(s-1)^2+(t-1)^2-2\right)=0.$$
The only possible positive integer solutions $(s,t)$ of the equation $(s-t)^2+(s-1)^2+(t-1)^2=2$ are $(2,2)$, $(2,1)$ and $(1,2)$. But $(s,t)=(2,2)$ is not of level $4$ because we can always find a subword which is fixed by $\theta_3$.
\par
As a conclusion, the only possible level in this case is $3$ and all possible relations coming from level $3$ elements are quantized Serre relations as shown above.

\subsection{Primitivity of Serre relations}
As an application of the main theorem, we deduce a short proof for the primitivity of Serre relations with little computation. A direct proof can be found in the appendix of \cite{AS00}.
\par
Let $A=(a_{ij})_{n\times n}=DC$ be a symmetrized Cartan matrix and $V$ be a $\mathbb{Z}^n$-Yetter-Drinfeld module of diagonal type with dimension $n$. Notations in the previous subsection are adopted. 
Moreover, suppose that the braiding matrix $(q_{ij})$ satisfies:
\begin{equation}\label{qij}
q_{ij}q_{ji}=q_{ii}^{c_{ij}},\ \ 1\leq i,j\leq n,
\end{equation}
and these $-c_{ij}$ are the smallest integers such that the equations (\ref{qij}) hold.

\begin{proposition}\label{primitives}
For any $1\leq i,j\leq n$, $i\neq j$, we denote $N=1-c_{ij}$. Then $P_{N+1}(v_i^{N}v_j)$ is a primitive element, where $P_{N+1}$ is the Dynkin operator.
\end{proposition}
\begin{proof}
At first, it is easy to show that
$$(1-\s_{N}^2\s_{N-1}\cdots\s_1)(v_i^Nv_j)=(1-q_{ii}^{-c_{ij}}q_{ij}q_{ji})v_i^Nv_j.$$
From the hypothesis (\ref{qij}) above, the right hand side is $0$, so from Proposition \ref{prop1}, $\theta_{N+1}(v_i^Nv_j)=v_i^Nv_j$.
\par
Moreover, it is obvious that for any $1<s<N+1$ and any positional embedding $\iota:\mathfrak{B}_s\hookrightarrow\mathfrak{B}_{N+1}$, $\iota(\theta_{s})(v_i^Nv_j)\neq v_i^Nv_j$. As a consequence, in the algorithm after Remark \ref{remark}, $X^{-1}$ is well defined and from the definition of $X$, $X^{-1}(v_i^Nv_j)=\lambda v_i^Nv_j$ for some non-zero constant $\lambda$. Then $P_{N+1}(v_i^Nv_j)$ is a nonzero solution of $S_{N+1}x=0$ of level $n$, so it is primitive by Theorem \ref{maintheorem}.
\end{proof}

\subsection{Quantized enveloping algebras revisited}
We keep notations in the beginning of Section \ref{qea}.
\par
As $A=DC$ is a symmetrized Cartan matrix, for any $1\leq i,j\leq n$, we have $d_ic_{ij}=d_jc_{ji}$. Then from the definition of $q_{ij}$, the following lemma is clear.
\begin{lemma}
Let $A=DC$ be a symmetrized Cartan matrix. Then for any $1\leq i,j\leq n$, 
$$q_{ij}q_{ji}=q_{ii}^{c_{ij}}.$$
\end{lemma}
Combined with Proposition \ref{primitives}, this lemma gives:
\begin{corollary}
Let $\g$ be a symmetrizable Kac-Moody algebra. Then degree $n$ quantized Serre relations in $\mathcal{U}_q(\g)$ are of level $n$. Moreover, the union of level $n$ elements for $n\geq 2$ generates $\mathfrak{I}(V)$ as an ideal.
\end{corollary}
\begin{remark}
The corollary above explains the reason for the importance of symmetrizable Kac-Moody algebras: they contain sufficient Serre relations. This gives a strong constraint on the representation theory of such Lie algebras.
\end{remark}

\section{Differential algebras of Nichols alegbras}\label{pairing}

The first part of this section is devoted to the generalization of some results in \cite{Fang10}, then we recall the construction of a pairing between two Nichols algebras.

\subsection{Pairings between Nichols algebras}
In this subsection, we want to recall a result of \cite{Chen07} and \cite{Mas09}. It should be remarked that these two constructions, though in different languages (one is dual to the other), are essentially the same.
\par
Let 
$$H=\bigoplus_{n=0}^\infty H_n,\ \ B=\bigoplus_{n=0}^\infty B_n$$
be two graded Hopf algebras with finite dimensional graded components.
\begin{definition}
A generalized Hopf pairing $\phi:H\times B\ra\mc$ is called graded if for any $i\neq j$, $\phi(H_i,B_j)=0$.
\end{definition}
\indent
We fix a graded Hopf pairing $\phi_0:H\times B\ra\mc$ between $H$ and $B$ and assume moreover that $\phi_0$ is non-degenerate.
\par
Let $V\in{}_H^H\mathcal{YD}$ and $W\in{}_B^B\mathcal{YD}$ be two Yetter-Drinfel'd modules, $\phi_1:V\times W\ra\mc$ be a non-degenerate bilinear form such that for any $h\in H$, $b\in B$, $v\in V$ and $w\in W$,
\begin{equation}\label{eq1}
\phi_1(h.v,w)=\sum\phi_0(h,w_{(-1)})\phi_1(v,w_{(0)}),
\end{equation}
\begin{equation}\label{eq2}
\phi_1(v,b.w)=\sum\phi_0(v_{(-1)},b)\phi_1(v_{(0)},w),
\end{equation}
where $\delta_V(v)=\sum v_{(-1)}\ts v_{(0)}$ and $\delta_W(w)=\sum w_{(-1)}\ts w_{(0)}$ are $H$-comodule and $B$-comodule structure maps, respectively.
\par
Let $T(V)$, $T(W)$ be the corresponding braided tensor Hopf algebras and $\fn(V)$, $\fn(W)$ be Nichols algebras associated to $V$ and $W$, respectively. Let $\fb_H(V)=\fn(V)\sharp H$ and $\fb_B(W)=\fn(W)\sharp B$ denote crossed biproducts defined in Section \ref{biproduct}.
\begin{theorem}[\cite{Chen07},\cite{Mas09}]\label{Chen}
There exists a unique graded Hopf pairing
$$\phi:\fb_H(V)\times\fb_B(W)\ra\mc,$$
extending $\phi_0$ and $\phi_1$. Moreover, it is non-degenerate.
\end{theorem}

\indent
In the following argument, attention will be paid to a particular case of this theorem. In our framework, we take $H=B$ and $V=W$ in Theorem \ref{Chen}, $\phi_0:H\times H\ra\mc$ a non-degenerate graded Hopf pairing and $\phi_1:V\times V\ra\mc$ a non-degenerate bilinear form satisfying the compatibility conditions above. So the machinery in Theorem \ref{Chen} produces a non-degenerate graded Hopf pairing
$$\phi:\fb_H(V)\times \fb_H(V)\ra\mc.$$
This will be the main tool in our further construction.

\subsection{Double construction and Schr\"odinger representation}\label{7.2}
In this subsection, as a review, we will apply results from \cite{Fang10}, Section 2 to the case of Nichols algebras.
\par
Suppose that $H$ is a graded Hopf algebra, $V\in{}_H^H\mathcal{YD}$ is an $H$-Yetter-Drinfel'd module and $\phi:\fb_H(V)\times\fb_H(V)\ra\mc$ is the non-degenerate graded Hopf pairing constructed in the last section.
\par
We recall some results from \cite{Fang10} briefly.
\par
To indicate their positions, we denote $\fb^+_H(V)=\fb_H(V)$, $\fb^-_H(V)=\fb_H(V)$ and
$$D_\phi(\fb_H(V))=D_\phi(\fb^+_H(V),\fb^-_H(V))$$
their quantum double. The Schr\"odinger representation defined in \cite{Fang10} gives a module algebra type action of $D_\phi(\fb_H(V))$ on these two components.
\begin{enumerate}
\item On $\fb_H^+(V)$, the action is given by: for $a,x\in\fb^+_H(V)$ and $b\in\fb^-_H(V)$,
$$(a\ts 1).x=\sum a_{(1)}xS(a_{(2)}),$$
$$(1\ts b).x=\sum \vp(x_{(1)},S(b))x_{(2)}.$$
\item On $\fb_H^-(V)$, the action is given by: for $a\in\fb^+_H(V)$ and $b,y\in\fb^-_H(V)$,
$$(a\ts 1).y=\sum\vp(a,y_{(1)})y_{(2)},$$
$$(1\ts b).y=\sum b_{(1)}yS(b_{(2)}).$$
\end{enumerate}
As has been shown in \cite{Fang10}, these actions give both $\fb^+_H(V)$ and $\fb^-_H(V)$ a $D_\phi(\fb_H(V))$-module algebra structure.
\par
Moreover, we can construct the Heisenberg double 
$$H_\phi(\fb_H(V))=H_\phi(\fb^+_H(V),\fb^-_H(V)),$$
which, in general, is not a Hopf algebra.
\par
In \cite{Fang10}, we defined an action of $D_\phi(\fb_H(V))$ on $H_\phi(\fb_H(V))$ by: for $a,a'\in\fb_H^+(V)$ and $b,b'\in\fb_H^-(V)$,
$$(a\ts b).(b'\sharp a')=\sum (a_{(1)}\ts b_{(1)}).b'\sharp (a_{(2)}\ts b_{(2)}).a',$$ 
which makes $H_\phi(\fb_H(V))$ a $D_\phi(\fb_H(V))$-module algebra.
\par
The following two results are also obtained in \cite{Fang10}.
\begin{proposition}[\cite{Fang10}]\label{17}
We define a $D_\phi(\fb_H(V))$-comodule structure on $\fb^+_H(V)$ and $\fb^-_H(V)$ by: for $a\in\fb_H^+(V)$ and $b\in\fb_H^-(V)$,
$$\fb^+_H(V)\ra D_\phi(\fb_H(V))\ts\fb_H^+(V),\ \ a\mapsto\sum a_{(1)}\ts 1\ts a_{(2)},$$
$$\fb^-_H(V)\ra D_\phi(\fb_H(V))\ts\fb_H^-(V),\ \ b\mapsto\sum 1\ts b_{(1)}\ts b_{(2)}.$$
Then with the Schr\"odinger representation and comodule structures defined above, both $\fb_H^+(V)$ and $\fb_H^-(V)$ are in the category ${}_{D_\phi}^{D_\phi}\mathcal{YD}$.
\end{proposition}
Moreover, the Heisenberg double $H_\phi(\fb_H(V))$ is in the $D_\phi(\fb_H(V))$-Yetter-Drinfel'd module category.
\begin{theorem}[\cite{Fang10}]
We define a $D_\phi(\fb_H(V))$-comodule structure on $H_\phi(\fb_H(V))$ by: for $a\in\fb^+_H(V)$ and $b\in\fb^-_H(V)$,
$$H_\phi(\fb_H(V))\ra D_\phi(\fb_H(V))\ts H_\phi(\fb_H(V)),$$
$$b\sharp a\mapsto \sum  (1\ts b_{(1)}).(a_{(1)}\ts 1)\ts b_{(2)}\sharp a_{(2)}.$$
Then with the module structure defined above and this comodule structure, $H_\phi(\fb_H(V))$ is in the category ${}_{D_\phi}^{D_\phi}\mathcal{YD}$.
\end{theorem}

\subsection{Construction of differential algebras}\label{consda}
In this section, we want to construct the differential algebra of a Nichols algebra. It generalizes the construction of quantized Weyl algebra in \cite{Fang10}.
\par
But it should be remarked that the construction in \cite{Fang10} concentrates on a specific Hopf algebra, say $\mc[\mathbb{Z}^n]$ and a special action on the Nichols algebra. So to generalize it, we need some more work.
\par
Let $\fn^+(V)$ and $\fn^-(V)$ be Nichols algebras contained in $\fb_H^+(V)$ and $\fb^-_H(V)$ respectively. We would like to give both $\fn^+(V)$ and $\fn^-(V)$ a $D_\phi(\fb_H(V))$-Yetter-Drinfel'd module algebra structure.
\par
For $\fn^+(V)$, the $D_\phi(\fb_H(V))$-module structure is given by the Schr\"odinger representation and the $D_\phi(\fb_H(V))$-comodule structure is given by:
$$\fn^+(V)\ra D_\phi(\fb_H(V))\ts \fn^+(V),$$
$$b\mapsto \sum (b_{(1)}\sharp (b_{(2)})_{(-1)}\ts 1)\ts (b_{(2)})_{(0)}.$$
This is obtained from the formula in Proposition \ref{17}.
\begin{proposition}
With the structures defined above, $\fn^+(V)$ is a $D_\phi(\fb_H(V))$-Yetter-Drinfel'd module algebra.
\end{proposition}
\begin{proof}
At first, we need to show that the Schr\"odinger representation preserves $\fn^+(V)$.
\par
Let $a\in\fb_H(V)$. It suffices to prove that if $x\in\fn^+(V)$, then both $(a\ts 1).x$ and $(1\ts a).x$ are contained in $\fn^+(V)$. For this purpose, because the action is linear, we write $a=b\sharp h\in\fn^+(V)\sharp H$. From the formula given in Radford's crossed biproduct,
\begin{eqnarray*}
((b\sharp h)\ts 1).x &=& \sum (b\sharp h)_{(1)}xS((b\sharp h)_{(2)})\\
&=& \sum (b_{(1)}\sharp (b_{(2)})_{(-1)}h_{(1)})(x\sharp 1)S((b_{(2)})_{(0)}\sharp h_{(2)})\\
&=&
\sum (b_{(1)}((b_{(2)})_{(-3)}h_{(1)}.x)\sharp(b_{(2)})_{(-2)}h_{(2)})(1\sharp S(h_{(3)})S((b_{(2)})_{(-1)}))(S((b_{(2)})_{(0)})\sharp 1)\\
&=& 
\sum (b_{(1)}((b_{(2)})_{(-3)}h_{(1)}.x)\sharp (b_{(2)})_{(-2)}h_{(2)}S(h_{(3)})S((b_{(2)})_{(-1)}))(S((b_{(2)})_{(0)})\sharp 1)\\
&=& 
\sum b_{(1)}((b_{(2)})_{(-1)}h_{(1)}.x)S((b_{(2)})_{(0)})\sharp 1,
\end{eqnarray*}
which is in $\fn^+(V)$.
\par
For the other action, we have:
$$(1\ts (b\sharp h)).x=\sum \phi(x_{(1)},S(b\sharp h))x_{(2)}.$$
From the definition of the crossed biproduct, when restricted to $\fn^+(V)$, the coproduct gives $\Delta:\fn^+(V)\ra \fb_H(V)\ts\fn^+(V)$, so the result is in $\fn^+(V)$.
\par
Thus the action and coaction of $D_\phi(\fb_H(V))$ on $\fn^+(V)$ are both well defined and as a consequence, $\fn^+(V)$ is a $D_\phi(\fb_H(V))$-module algebra.
\par
These structures are compatible because we have seen that the coaction defined above is just the restriction of the coproduct in $D_\phi(\fb_H(V))$ on $\fn^+(V)$.
\end{proof}
\indent
The same argument, once applied to $\mathfrak{N}^-(V)$, implies that $\fn^-(V)$ is a $D_\phi(\fb_H(V))$-Yetter-Drinfel'd module algebra.
\par
As remarked after the definition of Yetter-Drinfel'd modules, we may use the natural braiding in the category ${}_{D_\phi}^{D_\phi}\mathcal{YD}$ to give $\fn^-(V)\ts\fn^+(V)$ an associative algebra structure, which is denoted by $W_\phi(V)$ and is called the differential algebra of the Nichols algebra $\fn(V)=\fn^-(V)$. 
\par
This gives a natural action of $W_\phi(V)$ on $\fn^-(V)$, where $\fn^+(V)\subset W_\phi(V)$ acts by "differential".
\begin{remark}
\begin{enumerate}
\item
It should be pointed out that $\fn^-(V)\ts\fn^+(V)$ is a subalgebra of $H_\phi(\fb_H(V))$. This can be obtained from the definition of the braiding $\sigma$ in the category ${}_{D_\phi}^{D_\phi}\mathcal{YD}$ and the formula for the action of $D_\phi(\fb_H(V))$ on $\fb_H^-(V)$.Moreover, it is exactly the subalgebra of $H_\phi(\fb_H(V))$ generated by $\fn^-(V)$ and $\fn^+(V)$.
\item
This action of $W_\phi(V)$ on $\fn^-(V)$ can be explained as follows: we consider the trivial $\fn^+(V)$-module $\mc$ given by the counit $\ve$, then
$$Ind_{\fn^+(V)}^{W_\phi(V)}(\mc.1)=W_\phi(V)\ts_{\fn^+(V)}\mc.1$$
is isomorphic to $\fn^-(V)$ as a vector space and from this, $\fn^-(V)$ can be regarded as a $W_\phi(V)$-module.
\end{enumerate}
\end{remark}

\subsection{Non-degeneracy assumption}\label{nondeg}
We should point out that results in previous sections do not depend on the non-degeneracy of the generalized Hopf pairing. Some results concerned with this property will be discussed in this subsection. 
\par
Recall the notation $\fn(V)=\fn^-(V)$. So $W_\phi(V)$ acts on $\fn(V)$ by: for $x\in\fn^+(V)\subset W_\phi(V)$ and $y\in\fn(V)$,
$$x.y=\sum\phi(x,y_{(1)})y_{(2)}.$$
\indent
Because the generalized Hopf pairing is graded and non-degenerate, results in \cite{Fang10} can be generalized to the present context.
\par
Let $v_1,\cdots,v_n$ be a basis of $V$.
\begin{lemma}
Let $y\in\fn(V)$, $y\notin\mc^*$ such that for any basis element $v_i\in\fn^+(V)\subset W_\phi(V)$, $v_i.y=0$. Then $y=0$.
\end{lemma} 
\begin{proposition}\label{nondegenerate}
Let $y\in\fn(V)$ such that $y\neq 0$. Then there exists $x\in\fn^+(V)\subset W_\phi(V)$ such that $x.y$ is a non-zero constant.
\end{proposition}

\subsection{Explicit construction of pairing}
Let $V\in{}^H_H\mathcal{YD}$ be a finite dimensional Yetter-Drinfel'd module. Then its linear dual $V^*\in{}^H_H\mathcal{YD}$ is also a Yetter-Drinfel'd module such that the evaluation map $\text{ev}:V^*\ts V\ra\mc$ is in ${}^H_H\mathcal{YD}$ (see \cite{AS02}, Section 1.2).
\par
As both $V$ and $V^*$ are braided vector spaces with braidings coming from the Yetter-Drinfel'd module structures, both $T(V)$ and $T(V^*)$ are braided Hopf algebras. The canonical pairing $V^*\ts V\ra\mc$ given by the evaluation map $f\ts v\mapsto f(v)$ extends to a generalized pairing
$$\phi:T(V^*)\ts T(V)\ra \mc.$$
Moreover, the radicals of this pairing are $\mathfrak{I}(V^*)$ and $\mathfrak{I}(V)$, respectively (\cite{AG98}, Example 3.2.23, Definition 3.2.26), so it descends to a non-degenerate generalized pairing 
$$\phi:\mathfrak{N}(V^*)\ts \mathfrak{N}(V)\ra\mc.$$
This is the main ingredient we will use to prove the main theorem.

\section{Applications to Nichols algebras}
At first, we want to show that the differential algebra above generalizes the skew-derivation defined by N.D. Nichols \cite{Nic78}, explicitly used in M. Gra\~na \cite{Gra00} and I. Heckenberger in \cite{Hec08} for Nichols algebras of diagonal type.
\subsection{Derivations}\label{derivation}
We keep notations from previous sections, fix a basis $v_1,\cdots,v_n$ of $V$ and a dual basis $v_1^*,\cdots,v_n^*$ with respect to the evaluation map.
\begin{definition}
For $a\in T(V^*)$, we define the left derivation $\partial_a^L:T(V)\ra T(V)$ by: for $y\in\fn(V)$, 
$$\partial_a^L(y)=\sum \phi(a,y_{(1)})y_{(2)}.$$
If $a=v_i^*$, the notation $\partial_i^L$ is adopted for $\p_{v_i^*}^L$.
\end{definition}
The left derivation $\p_a^L$ descends to $\fn(V)$ and gives $\p_a^L:\fn(V)\ra\fn(V)$.
\par
In the proposition below, we suppose that the Nichols algebra is of diagonal type.
\begin{proposition}
For $i=1,\cdots,n$, the definition of $\p_i^L$ above coincides with the one given in \cite{Hec08}.
\end{proposition}
\begin{proof}
From the definition of $\p_i^L$, 
$$\p_i^L(v_{i_1}\cdots v_{i_k})=\sum\phi(v_i^*,(v_{i_1}\cdots v_{i_k})_{(1)})(v_{i_1}\cdots v_{i_k})_{(2)}.$$
Then after the definition of the coproduct in Nichols algebras, the fact that $\sigma$ is of diagonal type and $\phi$ is graded, we obtain that the terms satisfying 
$$\phi(v_i^*,(v_{i_1}\cdots v_{i_k})_{(1)})\neq 0$$
are contained in those given by the shuffle action $\mathfrak{S}_{1,k-1}$ in the coproduct formula. A simple calculation shows that 
$$\mathfrak{S}_{1,k-1}=\{1,\s_1,\s_1\s_2,\cdots,\s_1\cdots\s_{k-1}\},$$
which gives exactly the formula in the definition after Heckenberger.
\end{proof}
\begin{remark}
The advantage of our definition for differential operators on Nichols algebras are twofold:
\begin{enumerate}
\item This is a global and functorial construction, we never need to work in a specific coordinate system at the beginning;
\item We make no assumption on the type of the braiding, it has less restriction and can be applied to more general cases, for example: Hecke type, quantum group type, and so on.
\end{enumerate}
\end{remark}

\begin{remark}
In the same spirit, for $a\in T(V^*)$, the right differential operator $\p_a^R$ can be similarly defined by considering the right action: for $y\in T(V)$, 
$$\p_a^R(y)=\sum y_{(1)}\phi(a,y_{(2)}).$$
The left derivation $\p_a^R$ descends to $\fn(V)$ and gives $\p_a^R:\fn(V)\ra\fn(V)$.
\end{remark}
Some results from \cite{Hec08} can be generalized with simple proofs.
\begin{lemma}
Let $x\in T(V)$ and $a\in T(V^*)$. Then:
$$\Delta(\p_a^L(x))=\sum\p_a^L(x_{(1)})\ts x_{(2)},\ \ 
\Delta(\p_a^R(x))=\sum x_{(1)}\ts\p_a^R(x_{(2)}).$$
\end{lemma}
\begin{proof}
We prove it for $\p_a^L$:
$$\Delta(\p_a^L(x))=\Delta\left(\sum\phi(a,x_{(1)})x_{(2)}\right)=\sum\phi(a,x_{(1)})x_{(2)}\ts x_{(3)}=\sum\p_a^L(x_{(1)})\ts x_{(2)}.$$
\end{proof}
Moreover, we have following results:
\begin{lemma}
For any $a,b\in T(V^*)$, $\p_a^L\p_b^R=\p_b^R\p_a^L$.
\end{lemma}
\begin{lemma}
For $x,y\in T(V)$ and $1\leq i\leq n$, we have:
$$\p_i^L(xy)=\p_i^L(x)y+\sum x_{(0)}\phi((v_i^*)_{(-1)},x_{(-1)})\p_{(v_i^*)_{(0)}}^L(y).$$
\end{lemma}
\begin{proof}
At first, from the definition, 
$$\Delta(xy)=\sum x_{(1)}((x_{(2)})_{(-1)}.y_{(1)})\ts (x_{(2)})_{(0)}y_{(2)}.$$
so the action of $\p_i^L$ gives:
\begin{eqnarray*}
\p_i^L(xy) &=& \sum\phi(v_i^*,x_{(1)}((x_{(2)})_{(-1)}.y_{(1)}))(x_{(2)})_{(0)}y_{(2)}\\
&=& \sum\left(\phi(v_i^*,x_{(1)})\ve((x_{(2)})_{(-1)}.y_{(1)})+\ve(x_{(1)})\phi(v_i^*,(x_{(2)})_{(-1)}.y_{(1)})\right)(x_{(2)})_{(0)}y_{(2)}\\
&=& \sum\phi(v_i^*,x)y+\phi(v_i^*,x_{(-1)}.y_{(1)})x_{(0)}y_{(2)}\\
&=& \p_i^L(x)y+\sum\phi((v_i^*)_{(-1)},x_{(-1)})x_{(0)}\p_{(v_i^*)_{(0)}}^L(y).
\end{eqnarray*}
\end{proof}

\subsection{Taylor Lemma}
This subsection is devoted to generalizing the Taylor Lemma in \cite{Jos81} to Nichols algebras of diagonal type. We keep notations in the last subsection and suppose that $G=\mathbb{Z}^n$.
\par
For any $x\in T(V^*)$, the left derivation $\p_x^L:T(V)\ra T(V)$ will be denoted by $\p_x$ in this subsection.
\par
From now on, we fix a homogeneous primitive element $w\in T(V^*)$ of degree $\alpha$ (the degree given by $\mathbb{Z}^n$), denote $q_{\alpha,\alpha}=\chi(\alpha,\alpha)$ and
$$T(V)^{\p_w}=\{v\in T(V)|\ \p_w(v)=0\}.$$
\begin{remark}
It is easy to see that if $w\in T(V^*)$ is a non-constant element, $\p_w$ is a locally nilpotent linear map as it decreases the degree when acting on an element.
\end{remark}
\begin{lemma}[Taylor Lemma]
Suppose that $q$ is not a root of unity. If there exists some homogeneous element $a\in T(V)$ such that $\p_w(a)=1$, then $a$ is free over $T(V)^{\p_w}$ and as vector spaces, we have:
$$T(V)=T(V)^{\p_w}\ts_\mc \mc[a].$$
\end{lemma}
\begin{proof}
Recall that $\alpha$ is the degree of $w$.
\par
As $T(V)$ is $\mathbb{Z}^n$-graded and $\p_w$ is a linear map of degree $-\alpha$, we may suppose that $a$ is homogeneous of degree $\alpha$.
\par
It is clear that $T(V)^{\p_w}\ts\mc[a]\subset T(V)$. Now we prove the other inclusion.
\par
As we are working under the diagonal hypothesis, a simple computation gives that
$$\p_w^n(a^n)=(n)_{q_{\alpha,\alpha}}!.$$
Then if for some $x_i\in T(V)^{\p_w}$, $\sum x_ia^i=0$, applying $\p_w$ sufficiently many times will force all $x_i$ to be zero.
\par
Let $x\in T(V)$ be a homogeneous element. Then so is $\p_w(x)$. We let $\mu$ denote the degree of $x$ and $n\in\mathbb{N}$ a positive integer such that $\p_w^n(x)\neq 0$ but $\p_w^{n+1}(x)=0$. If $x\in T(V)^{\p_w}$, the lemma is proved. Now we suppose that $x\notin T(V)^{\p_w}$, which implies that $n>0$.
\par
We let $\lambda$ denote the degree of $\p_w^n(x)$ and $q_{\alpha,\lambda}=\chi(\alpha,\lambda)$, then
$$\p_w^n(\p_w^n(x)a^n)=q_{\alpha,\lambda}^n\p_w^n(x)\p_w^n(a^n).$$
If we define
$$X=x-\frac{1}{(n)_{q_{\alpha,\alpha}}!}\frac{1}{q_{\alpha,\lambda}^n}\p_w^n(x)a^n,$$
then $X\equiv x$ (mod $T(V)^{\p_w}\ts\mc[a]$) and
$$\p_w^n(X)=\p_w^n(x)-\frac{1}{(n)_{q_{\alpha,\alpha}}!}\frac{1}{q_{\alpha,\lambda}^n}q_{\alpha,\lambda}^n\p_w^n(x)\p_w^n(a^n)=0.$$
So the lemma follows by induction on the nilpotent degree of $x$.
\end{proof}
For a homogeneous element $b\in T(V)$, we let $\p^\circ b$ denote its degree.
\begin{remark}
Let $w\in T(V^*)$ be an element such that it has non-zero image in $\mathfrak{N}(V)$. As the pairing we are considering is non-degenerate, the element $a$ always exists if $q$ is not a root of unity.
\end{remark}
We denote $q_{ii}=\chi(\alpha_i,\alpha_i)$ and suppose that $q$ is not a root of unity.
\begin{theorem}
Let $w=v_i^*$ for some $1\leq i\leq n$ and $a\in T(V)$ be a homogeneous element satisfying $\p_i(a)=1$. We dispose $-\p^\circ v_i$ the degree of $a$ and define the adjoint action of $\mc[a]$ on $T(V)^{\p_i}$ by: for a homogeneous element $b\in T(V)^{\p_i}$, 
$$a\cdot b=ab-\chi(\p^\circ a,\p^\circ b)ba.$$
Then we can form the crossed product of $\mc[a]$ and $T(V)^{\p_i}$ with the help of this action, which is denoted by $T(V)^{\p_i}\sharp\mc[a]$. With this construction, the multiplication gives an isomorphism of algebra:
$$T(V)\cong T(V)^{\p_i}\sharp\mc[a].$$
\end{theorem}
\begin{proof}
At first, we should show that the action defined above preserves $T(V)^{\p_i}$. Let $b\in T(V)^{\p_i}$ and denote $q_{i,b}=\chi(\p^\circ v_i,\p^\circ b)$. Then
\begin{eqnarray*} 
\p_i(a\cdot b) &=& \p_i(ab)-q_{i,b}^{-1}\p_i(ba)\\ 
&=& \p_i(a)b+\chi(\p^\circ v_i,\p^\circ a)a\p_i(b)-q_{i,b}^{-1}\p_i(b)a-q_{i,b}^{-1}q_{i,b}b\p_i(a)\\
&=& \p_i(a)b-b\p_i(a)=0,
\end{eqnarray*}
where equations $\p_i(b)=0$ and $\p_i(a)=1$ are used. So the crossed product is well defined.
\par
Note that $a$ is primitive. We proceed to prove that the multiplication is an algebra morphism: what needs to be demonstrated is that for any $m,n\in\mathbb{N}$ and homogeneous elements $x,y\in T(V)^{\p_w}$, $$(x\ts a^m)(y\ts a^n)=xa^mya^n.$$
From definition, it suffices to show that
$$(1\ts a^m)(y\ts 1)=a^my.$$
\indent
At first, it should be pointed out that from the definition, the action of $\mc[a]$ on $T(V)^{\p_i}$ is just the commutator coming from a braiding. If we let $\Phi$ denote the linear map in Theorem \ref{theorem1}, then:
$$(1\ts a^m)(y\ts 1)=(\Phi\ts id)(\Delta(a^m)(y\ts 1)).$$
So it suffices to prove that
$$m\circ(\Phi\ts id)(\Delta(a^m)(y\ts 1))=a^my.$$
We proceed to show this by induction.
\par
For $m=1$, 
\begin{eqnarray*}
m\circ(\Phi\ts id)(\Delta(a)(y\ts 1)) &=& m\circ(\Phi\ts id)(ay\ts 1+\chi(\p^\circ a,\p^\circ y)y\ts a)\\
&=& ay-\chi(\p^\circ a,\p^\circ y)ya+\chi(\p^\circ a,\p^\circ y)ya\\
&=& ay.
\end{eqnarray*}
\indent
For the general case, we denote $\Delta(a^{m-1})(y\ts 1)=\sum x'\ts x''$, then
\begin{eqnarray*}
&\ &m\circ(\Phi\ts id)(\Delta(a^m)(y\ts 1))\\
&=& m\circ(\Phi\ts id)\left((a\ts 1+1\ts a)\left(\sum x'\ts x''\right)\right)\\
&=& m\circ(\Phi\ts id)\left(\sum ax'\ts x''+\sum \chi(\p^\circ a,\p^\circ x')x'\ts ax''\right)\\
&=& \sum a\Phi(x')x''-\sum \chi(\p^\circ a,\p^\circ x')\Phi(x')ax''+\sum \chi(\p^\circ a,\p^\circ x')\Phi(x')ax''\\
&=& \sum a\Phi(x')x''\\
&=& a^my,
\end{eqnarray*}
where the last equality comes from the induction hypothesis.
\end{proof}
\begin{remark}
In the theorem, we need to take the opposite of the degree of $a$ because it acts as a differential operator, which has negative degree. But to get an isomorphism, a positive one is needed.
\end{remark}

\section{Primitive elements}
This last subsection is devoted to giving a proof of Theorem \ref{maintheorem}.
\par
At first, as in the construction of Section \ref{pairing}, let $\phi:T(V^*)\ts T(V)\ra\mc$ be a generalized pairing which descends to a non-degenerate pairing $\phi:\fn(V^*)\ts \fn(V)\ra\mc$ between Nichols algebras.
\par
For $x\in T(V)$, we let $\Delta_{i,j}(x)$ denote the component of $\Delta(x)$ of bidegree $(i,j)$.
\begin{proposition}\label{prop0}
Let $x\in\ker(S_n)$ be a non-zero solution of equation $S_nx=0$ of level $n$. Then for any $i=1,\cdots,n$, $\p_i^R(x)=0$.
\end{proposition}
\begin{proof}
From the definition of $\p_i^R$ and the fact that $\phi$ is graded, the possible non-zero terms in $\p_i^R(x)$ are those belonging to $\Delta_{n-1,1}(x)$ in $\Delta(x)$.
\par
From the definition of the coproduct, $\Delta_{n-1,1}$ corresponds to the action of the element 
$$\sum_{\sigma\in\mathfrak{S}_{n-1,1}}T_\sigma.$$
It is clear that 
$$\mathfrak{S}_{n-1,1}=\{1,\s_{n-1},\s_{n-1}\s_{n-2},\cdots,\s_{n-1}\cdots\s_1\},$$
so in fact, $\Delta_{n-1,1}$ corresponds to the part $T_n$ in the decomposition of $S_n$. The condition of $x$ being of level $n$ means that $T_nx=0$, thus
$$(id\ts\phi(v_i^*,\cdot))\circ\Delta_{n-1,1}(x)=0,$$
and so $\p_i^R(x)=0$.
\end{proof}

\begin{corollary}\label{cor2}
With the assumption in the last proposition, for any non-constant $a\in T(V^*)$, we have $\partial_a^R(x)=0$.
\end{corollary}

\begin{proposition}\label{prop5}
Let $x\in T(V)$ be a homogeneous element which is not a constant. If for any non-constant element $a\in T(V^*)$, $\p_a^R(x)=0$, then $\Delta(x)-x\ts 1\in T(V)\ts\mathfrak{I}(V)$ and $x\in\mathfrak{I}(V)$.
\end{proposition}

\begin{proof}
If for any $a\in T(V^*)$, $\p_a^R(x)=0$, then $\sum x_{(1)}\phi(a,x_{(2)})=0$ for any $a$. We choose $x_{(1)}$ to be linearly independent. If $x_{(2)}$ is not a constant, it must be in the right radical of $\phi$, which is exactly $\mathfrak{I}(V)$.
\par
So $\Delta(x)-x\ts 1\in T(V)\ts\mathfrak{I}(V)$. We obtain that $x\in\mathfrak{I}(V)$ by applying $\ve\ts id$ on both sides.
\end{proof}

It is clear that if $x\in\ker(S_2)$ be a non-zero solution of $S_2x=0$ with level $2$, then $x$ is primitive.

\begin{theorem}
Let $n\geq 2$ and $x\in\ker(S_n)$ be a non-zero solution of equation $S_nx=0$ with level $n$. Then $x$ is primitive and it is in $Im(P_n)$.
\end{theorem}
\begin{proof}
The case $n=2$ is clear.
\par
Let $n>3$ and $x$ be a solution of equation $S_nx=0$ with level $n$. It suffices to prove that $\Delta_{i,n-i}(x)=0$ for any $2\leq i\leq n-2$.
\par
From the definition, components in $\Delta_{i,n-i}(x)$ can be obtained by acting a shuffle element on $x$, we want to show that  $\Delta_{i,n-i}(x)=\sum x'\ts x''=0$.
\par
From Corollary \ref{cor2}, $x$ is of level $n$ implies that for any non-constant $a\in T(V^*)$, $\p_a^R(x)=0$. So from Proposition \ref{prop5}, $\Delta(x)-x\ts 1\in T(V)\ts \mathfrak{I}(V)$ and then $x''\in\mathfrak{I}(V)$, $S_{n-i}x''=0$. It is easy to see that there exists a positional embedding $\iota:\mathfrak{B}_{n-i}\hookrightarrow \mathfrak{B}_n$ such that 
$$\iota(S_{n-i})\left(\sum_{\sigma\in\mathfrak{S}_{i,n-i}}T_\s\right)(x)=0,$$
which means that the equation $\iota(S_{n-i})v=0$ has a non-zero solution in $\mc[X_x]$.  It contradicts Lemma \ref{2->1}.
\end{proof}
\begin{corollary}
Let $n\geq 2$ and $E_n$ be the set of level $n$ solutions of equation $S_nx=0$ in $V^{\ts n}$. Then $E_n$ is a subspace of $T(V)$. If we denote $P=\bigoplus_{n\geq 2}E_n$, then $P$ is a coideal and the ideal $K$ generated by $P$ in $T(V)$ is contained in $\mathfrak{I}(V)$.
\end{corollary}
$ $ \\
\begin{center}
List of notations\\
\begin{tabular}[t]{l|c|l|c}
\hline
Notation(s) & Section & Notation(s) & Section\\
\hline
$\mathcal{P}_{i,j}, P_{i,j}, \mathfrak{S}_{k,n-k}, T_w$ & \ref{Dynkin} &
$X, E_n$ & \ref{solutions}\\
$\Phi, \mathcal{N}$ & \ref{3.2} &
$\Pi_s^{\underline{k}}$ & \ref{diagonaltype}
\\
$\theta_n, \Delta_n$ & \ref{4.1} &
$\mathfrak{B}_H^+(V), \mathfrak{B}_H^-(V)$ & \ref{7.2}
\\
$S_n,T_n,P_n,T_n',L_n$ & \ref{decomposition} &
$\mathfrak{N}^+(V),\mathfrak{N}^-(V)$ & \ref{consda}
\\
$\mathbb{C}[X_v], \mathcal{H}$ & \ref{assumption} &
$\p_a^L, \p_a^R, \p_i^L, \p_i^R$ & \ref{derivation}
\\
\hline
\end{tabular}
\end{center}
$ $\\

\end{document}